\documentclass[12pt,a4paper]{amsart}
\usepackage{amsfonts,amsmath,amssymb,amsxtra,latexsym}

\newtheorem{teo}{Theorem}[section]
\newtheorem{prop}[teo]{Proposition}
\newtheorem{lemma}[teo]{Lemma}

\newtheorem{remark}{Remark}
\newtheorem{defi}{Definition}

\numberwithin{equation}{section}

\newcommand{\ed}{\end{document}}
\newcommand{\ep}{\end{proof}\ed}

\newcommand{\real}{\mathbb{R}}

\newcommand{\N}{\mathbb{N}}

\newcommand{\sta}{\stackrel}
\newcommand{\dsp }{\displaystyle}
\newcommand{\di}{{\displaystyle\int}}

\newcommand{\bv}{\mathbf{v}}
\newcommand{\bu}{\mathbf{u}}
\newcommand{\ba}{\mathbf{a}}
\newcommand{\bn}{\mathbf{n}}
\newcommand{\bb}{\mathbf{b}}
\newcommand{\bw}{\mathbf{w}}
\newcommand{\bphi}{\boldsymbol{\varphi}}
\newcommand{\bpsi}{\boldsymbol{\psi}}

\begin{document}

\title[Steady Flow for Shear Thickening Fluids]
{Steady Flow for 
Shear Thickening Fluids with Arbitrary Fluxes
}
\author[G.J. Dias]{Gilberlandio J. Dias}
\address{Colegiado de Matem\'atica, Universidade Federal do Amap\'a-UNIFAP\\
Rodovia Juscelino Kubistchek de Oliveira, s/n, Jardim Marco Zero\\ 
Caixa Postal 261\\
Macap\'a, AP 68902-280, Brazil}
\email{gjd@unifap.br}

\author[M.M. Santos]{Marcelo M. Santos}
\address{Departamento de Matem\'atica, IMECC\\
Rua S\'ergio Buarque de Holanda, 651, Cidade Universit\'aria Zeferino Vaz\\
Universidade Estadual de Campinas - UNICAMP\\
Campinas, SP 13083-859, Brazil}
\email{msantos@ime.unicamp.br}

\thanks{Supported partly by Conselho Nacional de Desenvolvimento Cient\'ifico e Tecnol\'ogico (CNPq), Brasil, under grant 307192/2007-5.
}



\maketitle
\begin{abstract} 
We solve the stationary Navier-Stokes equations for non-Newtonian incompressible fluids with shear dependent viscosty in domains with unbounded outlets, in the case of shear thickening viscosity, i.e. the
viscosity $\mu$ is given by the power law $\mu=|D(\bv)|^{p-2}$, where
$|D(\bv)|$ is the shear rate and $p>2$. The flux assumes arbitrary
given values and the Dirichlet integral of the velocity
field grows at most linearly in the outlets of the domain. 
Under some smallness conditions on the {\lq\lq}energy dispersion{\rq\rq}
we also show that the solution of this problem is unique. Our results
are an extension of those obtained in \cite{ls} for Newtonian fluids~ ($p=2$).
\end{abstract}

\maketitle

\section{Introduction}
\label{introd}
$\indent$ 
The Navier-Stokes system for stationary incompressible flows in a domain with
unbounded straight outlets, with the velocity field converging to parallel
flows (Poiseuille flow) in the ends of the outlets, was solved first by 
C. Amick \cite{amick} in the 1970s. This problem
is known as {\em Leray problem}, 
cf. \cite[p. 476]{amick}. Amick's solution assumes the fluxes of the fluid in the
outlets to be sufficiently small, which turns out to be a sufficient condition to deal with the convective 
(nonlinear) term in Navier-Stokes equations. It is an open problem to
solve Leray problem for arbitrary fluxes. Alternately, Ladyzhenskaya and Solonnikov
\cite{ls} 
considered the stationary Navier-Stokes equations not demanding the fluid to be parallel in the ends of the outlets, but instead having arbitrary fluxes. In this case, the
outlets do not need to be straight and they solved this new problem for domains having arbitrary
uniformly bounded cross sections and with the fluid having arbitrary fluxes. Besides, their solution has the property that the Dirichlet's integral of the velocity field of the fluid grows
at most linearly with the direction of each outlet, and they also proved that this solution is
unique under some additional smallness condition.
   
In this paper we extend the 
Ladyzhenskaya-Solonnikov's theorem, i.e. {\lq\lq}Theorem 3.1{\rq\rq} in
\cite{ls}, for power-law shear thickening fluids,
i.e. incompressible non-Newtonian fluids obeying the power law
\begin{equation}
\label{power-law}
\mathbb{S}=|D({\mathbf{v}})|^{p-2}D({\mathbf{v}}), 
\end{equation}
when $p>2$. Here, $\mathbb S$ is the viscous stress tensor, $\mathbf{v}$ is the velocity field
of the fluid and $D(\mathbf{v})$ is the symmetric part of velocity gradient $\mathbf{\nabla v}$
(i.e. $D_{ij}(\mathbf{v})=\frac{1}{2}(\frac{\partial v_j}{\partial x_i}+\frac{\partial v_i}{\partial x_j}$) for $\mathbf{v}=(v_1, \cdots, v_n)$, \ $i,j\in\{1,\cdots,n\}$, 
$n\in\mathbb{N}$).
For $p=2$, the fluid is Newtonian. If
$1<p<2$, the fluid is called {\em shear thinning} (or plastic and pseudo-plastic) and if $p>2$, {\em shear thickening} (or dilatant). 
In engineering literature the power law \eqref{power-law} is also known as 
Ostwald-De Waele law (see e.g. \cite{bsl}). 
Corresponding to \eqref{power-law} we have the following system of equations modelling the flow of an incompressible fluid in a stationary
regime:
\begin{equation}
\label{nspl}
\left\{\begin{array}{c}
-\,\mbox{div}(|D({\mathbf{v}})|^{p-2}D({\mathbf{v}})) + (\mathbf{v}\cdot\nabla)\mathbf{v} 
+\nabla {\mathcal P} = 0\\
\ \ \ \mbox{div}\,\mathbf{v}=0\, ,
\end{array}\right.
\end{equation}
where ${\mathcal P}$ is the pression function of the fluid (and $\mathbf{v}$ is the velocity
field, as already indicated above). 
This model equations are also referred to as Smagorinsky model, due
to \cite{smago}, or Ladyzhenskaya model, due to \cite{lady-2,lady-1,lady}.
A related model where the viscosity is given by $|\bv|^{p-2}$, instead of $|D(\bv)|^{p-2}$, is considered in \cite{lions}.  
%
For this case, it is shown in \cite[Remark 5.5 in Chap.2, \S 5.2]{lions} the existence of a (weak) solution for system \eqref{nspl} in a bounded domain with homogeneous
Dirichlet boundary condition, for $p\geq \frac{3n}{n+2}$ .  There are
many results concerning the solution of \eqref{nspl} in bounded domains.
For instance, in \cite{fms} the existence of a solution for \eqref{nspl}  is obtained under the weaker condition that $p\ge\frac{2n}{n+2}$.

In unbounded domains there are not so many results.
For parallel fluids we can identify $\mathbf{v}$ with a scalar
function $v$ and the system \eqref{nspl} reduces to the $p$\,-{\em Laplacian} equation
\begin{equation}
\label{parallel p-laplacian}
-\,\mbox{div}(|\nabla v|^{p-2}\nabla v) = c
\end{equation}
for some constant $c$ (related to the {\lq\lq}{\em pressure drop}{\rq\rq}).
So, we can consider the Leray problem for \eqref{nspl}, i.e. the solution of \eqref{nspl} in a domain with straight outlets with the velocity field
tending to the solution of \eqref{parallel p-laplacian} in the ends of
the outlets. This problem was solved by E. Maru\v{s}i\'c-Paloka \cite{marusic}
under the condition that the fluxes are sufficiently small and $p>2$, thus extending Amick's
theorem \cite{amick} for power fluids with $p\ge2$. As far as we know,
the Leray problem for \eqref{nspl} when $p<2$ (with small fluxes) is an
open problem.

In this paper, as we mentioned above, we extend Ladyzhenskaya-Solonnikov's
theorem \cite[Theorem 3.1]{ls} for \eqref{nspl} when $p>2$. More precisely,  we obtain the existence of a solution $\bv$ to the system
\eqref{nspl} for $n=2,3$, and $p\ge2$, in a domain $\Omega$ with unbounded outlets,
specified in the next Section, for any given fluxes in the outlets and homogeneous 
Dirichlet boundary condition $\bv|\partial\Omega=0$. 
The {\lq\lq}Dirichlet integrals{\rq\rq} $\int |\nabla\mathbf{v}|^p$ of our solution grows at most linearly with the direction of the outlets (see 
\eqref{lady-solo pb}$_5$ in Section \ref{section 2}).
Besides, we observe that these integrals over portions of the outlets
with a fixed {\lq}length{\rq} are bounded by a constant that tends to zero with the flux (see Proposition \ref{prop.2} and Remark \ref{kappa alpha}).
Under this condition and some aditional one, we have uniqueness of
solution (see Theorem \ref{teo.2}). All these facts were obtained in
\cite{ls} for the case $p=2$, but the power-law model (\eqref{nspl} with $p\not=2$) was not treated in \cite{ls}.
In the next two paragraphs we look at some facts relating to the case $p\not=2$.
 
First, 
to deal with the nonlinear term
$\mbox{div}(|D({\mathbf{v}})|^{p-2}D({\mathbf{v}}))$ one can use the monotone method of Browder-Minty.
Secondly, we extend the technique employed in \cite{ls} to obtain the existence of a solution, which, in particular, consists in first solving the problem
in a bounded truncated domain and then taking the limit when
the parameter of the truncation tends to infinity, to obtain a solution in the whole domain.
To take this limit we need first uniform estimates with respect to the
truncation parameter for the solution
in the truncated domain, and this is obtained by integrating by parts the equation times the solution in some fixed bounded domain. Then we need the regularity of
the solution in bounded domains, more precisely, that the solutions
have velocity field at least in the Sobolev space $W^{2,l}$
and pressure in $W^{1,l}$, for some positive number $l$, due to the
boundary terms that comes from the integration by parts.
However this regularity is not expected for the weak solutions of 
\eqref{nspl}, if $p\not=2$. To overcome this difficult, when 
dealing with \eqref{nspl} in a truncated bounded domain we modify it to
\begin{equation}
\label{nspl modified}
\left\{\begin{array}{c}
-\,\mbox{div}\{\big(\frac{1}{T}+|D(\bv)|^{p-2}\big)D(\bv)\} + (\bv\cdot\nabla)\bv +\nabla {\mathcal P} = 0\\
\ \ \ \mbox{div}\,\mathbf{v}=0\, ,
\end{array}\right.
\end{equation}
where $T>0$ is the truncation parameter. See Proposition \ref{prop.1}
in Section \ref{main results}. 



\smallskip

\smallskip

As in \cite{sp} and \cite{ls}, and in several subsequent papers, here the velocity field ${\bv}$ is sought in the form $\bv=\bu+\ba$, where
$\bu$ is the new unknown with zero flux and $\ba$ is a constructed vector
field carrying the given fluxes in the outlets (i.e. if the given flux in
an outlet with cross section $\Sigma$ is $\alpha$ then $\int_\Sigma\ba\cdot\bn=\alpha$ and $\int_\Sigma\bu\cdot\bn=0$, where $\bn$ is the unit normal vector to $\Sigma$ pointing toward infinity). This vector field
$\ba$ depends on the geometry of the domain and, in the aforementioned papers, its construction is very
tricky and makes use of the Hopf cutoff function (see \cite{sp,ls}). 
In the case of power-law fluids \eqref{nspl} with $p>2$ we found out that 
the construction of $\ba$ can be quite simplified. Indeed, a key point in
the construction, in any case, is to obtain a vector field $\ba$ that
controls the quadratic nonlinear term $(\bu\nabla\bu)\ba$,
which appears after substituting $\bv=\bu+\ba$ in \eqref{nspl} and
multiplying it by $\bu$. That is, to obtain a priori estimates, 
one multiplies the first equation in \eqref{nspl} by $\bu$ and try to
bound all the resulting terms by the {\lq}leading{\rq} term $|D(\bu)|^p$.
In \cite{ls} it is shown that for any positive number $\delta$ there is a vector field $\ba$ which, in particular, satisfies the estimate
$$
\int_{\Omega_t}|\bu|^{2}|\ba|^{2}
\le c\delta^2\int_{\Omega_t}|\nabla\bu|^2
$$
for some constant $c$ indepedent of $\delta$, $\bu$ and $\Omega_t$, where $\Omega_t$ is any truncaded portion of the domain with a length of order $t$. Looking at their construction and using Korn's inequality 
it is possible to show that
\begin{equation}
\label{delta estimate}
\int_{\Omega_t}|\bu|^{p'}|\ba|^{p'}\le c\delta^{p'}t^{(p-2)/(p-1)}\left(\int_{\Omega_t}|D(\bu)|^p\right)^{p'/p},
\end{equation}
where $p'$ is the conjugate exponent of $p$, i.e. $p'=p/(p-1)$. 
When $p=2$ this estimate reduces to $|\int_{\Omega_t}(\bu\nabla\bu)\ba|\le c\delta\int_{\Omega_t}|\nabla \bu|^{2}$. With this estimate we can estimate the integral of
$(\bu\nabla\bu)\ba$ in the truncated domain $\Omega_t$, by using H\"older 
inequality:
\begin{equation}
\label{convective delta estimate}
\begin{array}{rl}
|\di_{\Omega_t}(\bu\nabla\bu)\ba\,|
&\le\left(\di_{\Omega_t}|\nabla\bu|^2\right)^{1/2}
\left(\di_{\Omega_t}|\bu|^{2'}|\ba|^{2'}\right)^{1/2'}\\
&\le c\delta \di_{\Omega_t}|\nabla \bu|^{2}.
\end{array}
\end{equation}
Thus we can control the nonlinear term $(\bu\nabla\bu)\ba$ by taking {\bf
necessarily} $\delta$ sufficiently small. 
When $p>2$, proceeding similarly and using also Korn's inequality, we obtain
\begin{equation}
\label{convective delta estimate}
|\di_{\Omega_t}(\bu\nabla\bu)\ba\,|
\le c\delta t^{(p-2)/p}\left(\di_{\Omega_t}|D(\bu)|^{p}\right)^{2/p}.
\end{equation}
Then, by Young inequality with $\epsilon$, we have
$$
|\di_{\Omega_t}(\bu\nabla\bu)\ba\,|
\le \epsilon \di_{\Omega_t}|D(\bu)|^p+C_\epsilon t\, ,
$$
for some new constant $C_\epsilon$. From this estimate, we can control the
nonlinear term $(\bu\nabla\bu)\ba$ by taking $\epsilon$ sufficiently
small, and so we do not need to construct the vector field $\ba$ satisfying the estimate \eqref{delta estimate} for a sufficiently small $\delta$. See Section \ref{main results} for the details. In fact, if $\ba$ is only a (smooth) bounded divergence free vector field vanishing
on $\partial\Omega$, then, by Poincar\'e, H\"older
and Korn inequalities, and the fact that our domain has uniformly bounded cross sections and $p/p'=p-1>1$ ($p>2$), we have 
\begin{equation}
\label{1 estimate}
\begin{array}{rl}
\di_{\Omega_t}|\bu|^{p'}|\ba|^{p'} &\le c \di_{\Omega_t}|\bu|^{p'}
\le c \di_{\Omega_t}|\nabla \bu|^{p'}\\
&\le c t^{1-p'/p} \left(\di_{\Omega_t}|\nabla \bu|^{p}\right)^{p'/p}\\
&= c t^{(p-2)/(p-1)} \left(\di_{\Omega_t}|\nabla \bu|^{p}\right)^{p'/p}\\
&\le c t^{(p-2)/(p-1)} \left(\di_{\Omega_t}|D(\bu)|^{p}\right)^{p'/p}
\end{array}
\end{equation}
which is \eqref{delta estimate} for $\delta=1$. 
 
\bigskip

The plan of this paper is the following. Besides this introduction, in
Section \ref{section 2} we introduce the main notations and set precisely
the problem we will solve, state a lemma about the existence of the vector
field $\ba$, carrying the flux of the fluid, and state our main theorem
(Theorem \ref{main theorem}). In Section \ref{preliminary results} we
state some preliminaries results we need to prove our 
main results. In Section \ref{main results} we prove our main 
theorem, make some remarks and prove a result about the uniqueness of our
solution. 

\section{Ladyzhenskaya-Solonnikov problem for power-law fluids}
\label{section 2}

In this section we set notations and the problem we are concerned
with and state a lemma and our main theorem.

\smallskip

We denote by $\Omega$ a domain in $\real^n$, $n=2,3$,
with a $C^\infty$ boundary, of the following type:
$$\Omega=\bigcup_{i=0}^2\Omega_i\;,$$
where $\Omega_0$ is a bounded subset of $\real^n$, while, in 
different cartesian coordinate system,
$$
\Omega_1=\{x\equiv (x_1,x^\prime)\in\real^n;\;x_1<0,\; x^\prime
\in\Sigma_1(x_1)\}
$$ 
and
$$
\Omega_2=\{x\equiv(x_1,x^\prime)\in\real^n;\;x_1>0,\;x^\prime
\in\Sigma_2(x_1)\},
$$
with $\Sigma_i(x_1)$ being $C^\infty$ simply
connected domains (open sets) in $\real^{n-1}$, and such that, for constants $l_1,l_2$,
$0<l_1<l_2<\infty$, 
they sastify
$$
 \sup_{(-1)^i x_1>0}\mbox{diam}\,\Sigma_i(x_1)\le l_2
$$
and contain the cylinders
$$
\begin{array}{l}
C_{l_1}^i=\{x\in\real^n;\;(-1)^ix_1>0\;
\mbox{e}\;|x^\prime|<\frac{l_1}{2}\}\subset\Omega_i\, ;\\
\end{array}
$$
$i=1,2$.

\smallskip

For simplicity, we will denote by $\Sigma$ any of the cross sections $\Sigma_i\equiv\Sigma_i(x_1)$ or, more generaly, any cross section of $\Omega$,
i.e., any bounded intersection of $\Omega$ with a
$(n\!-\!1)$\,-dimensional plane. We will denote by $\mathbf{n}$, the ortonormal vector to $\Sigma$ pointing from $\Omega_1$ toward $\Omega_2$
i.e. in the above local coordinate systems, we have $\mathbf{n}=(1,{0})$ (where ${0}\in\real^{n-1}$) in
both outlets $\Omega_1$ and $\Omega_2$. 
With these notations, the flux through any cross section $\Sigma$  of $\Omega$ of an incompressible fluid
in $\Omega$ with velocity field $\bv$ vanishing on $\partial\Omega$, is given by the {\lq}surface{\rq}
integral $\int_{\Sigma}\bv\cdot\bn$ (notice that
by the divergence theorem applied to the region bounded by $\partial\Omega$,
$\Sigma_1$ and $\Sigma_2$, we have $\int_{\Sigma_1}\bv\cdot\bn=\int_{\Sigma_2}\bv\cdot\bn$, for any cross sections $\Sigma_1$ and
$\Sigma_2$ of $\Omega_1$ and $\Omega_2$, respectively).

\smallskip

We remark that we take our domain $\Omega$ with only two outlets $\Omega_i$, $i=1,2$, just to 
simplify the presentation, i.e. we can take $\Omega$ with any finite number of outlets with no significant change in the notations, results and
proofs given in this paper.

\smallskip

We shall use the further notations, where $U$ is an arbitrary subdomain of $\Omega$, $s>t>0$ and $1\le q<\infty$:
$$
\begin{array}{l}
\Omega_{i,t}=\{x\in\Omega_i\,;\,(-1)^i x_1<t\}\,,\;i=1,2\\
\Omega_{i,t,s}=\Omega_{i,s}\smallsetminus\overline {\Omega}_{i,t}\\
\Omega_t=\Omega_0\cup\Omega_{1,t}\cup\Omega_{2,t}\\
\Omega_{t,s}=\Omega_s\smallsetminus\overline{\Omega}_t\\
 \parallel\mathbf{v}\parallel_{q,U}=\left(\int_U|\mathbf{v}|^q\right)^{1/q}\\
 \parallel\mathbf{v}\parallel_{1,q,U}=\left(\int_U|\mathbf{v}|^q+
|\nabla\mathbf{v}|^q\right)^{1/q}\\
 |\mathbf{v}|_{1,q,U}=\left(\int_U|\nabla\mathbf{v}|^q\right)^{1/q}\\
\left(\bu,\bv\right)=\left(\bu,\bv\right)_U= \int_U\bu\cdot\bv\\
{\mathcal D}(U)=\{\bphi\in C_c^\infty(U;\mathbb{R}^n);\,\nabla\cdot\bphi=0\}\\
{\mathcal D}_0^{1,q}(U)=\overline{{\mathcal D}(U)}^{|\cdot|_{1,q}}
\end{array}
$$
In these notations, the set $\Omega_t$ - a bounded cut of $\Omega$
with a {\lq\lq}length{\rq\rq} of order $t$ - will be taken usually
for large $t$, so this notation will not cause confusion with the
(unbounded) outlets $\Omega_i$, where $i=1,2$. 

By $W^{1,q}(U)$ and $W_0^{1,q}(U)$ we stand for the usual Sobolev spaces, 
consisting of vector or scalar valued functions, and $W_{loc}^{1,q}(\overline{U})$ is the set of functions in $W^{1,q}(V)$ for
any bounded open set $V\subset U$.
Often when it is clear from the context we will omit the domain of integration in the notations.

The notation $|E|$ will stand for the Lebesgue measure of a Lebesgue
measurable set $E$ in the dimension which is clear in the context.
Finally, the same symbol $C$, $c$, $C_{\mathbf\cdot}$ or $c_{\mathbf\cdot}$ will denote many different constants.

\medskip

In this paper, we are concerned with the following problem: given $\alpha\in\mathbb{R}$, find a velocity field $\mathbf{v}$ and a pressure $\mathcal P$ such that

\begin{equation}
\label{lady-solo pb}
\left\{
\begin{array}{c}
\begin{array}{rll}
\mbox{div}\{|D(\mathbf{v})|^{p-2}D(\bv )\}&=&\mathbf{v}\cdot\nabla\mathbf{v}+\nabla
{\mathcal P}\;\;\;\;\mbox{in}\;\Omega\\
\nabla\cdot\mathbf{v}&=&0\;\;\;\;\mbox{in}\;\Omega\\
\mathbf{v}&=&0\;\;\;\;\mbox{on}\;\partial\Omega\\
 \int_\Sigma\mathbf{v}\cdot\mathbf{n}&=&\alpha
\end{array}
\\
 \sup_{t>0}t^{-1}\int_{\Omega_t}|\nabla\mathbf{v}|^p<\infty\,.
\end{array}
\right. 
\end{equation}
Cf. {\em Problem 2.1} in \cite{ls} (for the case $p=2$).
Here, and throughout, we use the notation
$$
 \mathbf{v}\cdot\nabla\mathbf{w}= (\mathbf{v}\cdot\nabla)\mathbf{w}
=\sum_{i=1}^nv_i\frac{\partial}{\partial x_i}\mathbf{w}
=(\bv\cdot\nabla w_1,\cdots,\bv\cdot\nabla w_n)
$$
for any velocity fields $\mathbf{v}=(v_1,\cdots,v_n)$ and 
$\mathbf{w}=(w_1,\cdots,w_n)$ defined in $\Omega$ such that
the last expression on the right makes sense.

\bigskip

To solve \eqref{lady-solo pb}, 
we seek a velocity field $\mathbf{v}$ in the form $\mathbf{v}={\bu}+{\ba}$,
where ${\bu}$ is a vector field with zero flux and {\bf a} will carry
the flux $\alpha$, i.e. $\int_{\Sigma}\bu\cdot\bn=0$ and $\int_{\Sigma}\ba\cdot\bn=\alpha$. More precisely, we shall take {\bf a}
to be a vector field having the properties given by the following lemma.

\begin{lemma}
\label{a} 
For any $p\ge2$, there exists a smooth divergence free vector field $\tilde\ba$, which is bounded and has bounded derivatives in $\overline{\Omega}$, vanishes on $\partial\Omega$, and has flux one, i.e. $\int_\Sigma\tilde\ba=1$ over any cross section $\Sigma$ of $\Omega$. 
In particular, given $\alpha\in\real$, the vector field
$\ba=\alpha\tilde\ba$ is a vector field preserving all these properties
but having flux $\alpha$ and else satisfying the following estimates:

{\rm i)} \ \ \  $\int_{\Omega_t}|\ba|^{p^\prime}|\mathbf{\mathbf\varphi}|^{p^\prime}\leq
c{|\alpha|}^{p^\prime}t^{(p-2)/(p-1)}|\mathbf{\mathbf\varphi}|_{1,p,\Omega_t}^{p^\prime}, \ \ \forall\,t>0,
\ \ \forall \ \mathbf{\mathbf\varphi}\in{\mathcal D}(\Omega)$;

{\rm ii)} \ \ $\int_{\Omega_{i,t-1,t}}|\nabla\ba|^p\leq c{|\alpha|}^p,
\ \ \forall\,t\geq 1,\,i=1,2$;

%
{\rm iii)} \ $\int_{\Omega_t}|\nabla{\ba}|^p\leq c{|\alpha|}^p(t+1)\;,\ \ \forall\,t\geq1$;

\noindent
where $p'=p/(p-1)$ and $c$ is a constante depending only on $\tilde\ba$, $p$ and~$\Omega$.
\end{lemma}
%

The proof of this lemma is given in Section \ref{main results}.

\bigskip

For $\mathbf{v}={\bu}+{\ba}$, problem \eqref{lady-solo pb} becomes

\begin{equation}
\label{ls for u}
\left\{\begin{array}{l}
\mbox{div}\{|D({\bu})+D({\ba})|^{p-2}(D({\bu})
+D({\ba}))\}\\
={\bu}\cdot\nabla{\bu}+{\bu}\cdot\nabla{\ba}
+{\ba}\cdot\nabla{\bu}
+{\ba}\cdot\nabla{\ba}+\nabla
{\mathcal P}\;\;\;\;\mbox{in}\;\Omega\\
\nabla\cdot{\bu}= 0\;\;\;\;\mbox{in}\;\Omega\\
{\bu}=0\;\;\;\;\mbox{on}\;\partial\Omega\\
 \int_\Sigma{\bu}\cdot\mathbf{ n}=0
\\
 \sup_{t>0}t^{-1}\int_{\Omega_t}|\nabla{\bu}|^p<\infty\,.
\end{array}\right.
\end{equation}

\noindent
Formally, multiplying \eqref{ls for u}$_1$ by $\bphi=(\varphi_1,\cdots,\varphi_n)\in{\mathcal D}(\Omega)$ and noticing that
\begin{equation}
\sum_{i,j=1}^n D({\bu})_{ij}\frac{\partial \varphi_i}{\partial
x_j}=\sum_{i,j=1}^n D({\bu})_{ij}D(\bphi)_{ij}
\end{equation}
and $\int_\Omega \nabla {\mathcal P}\cdot\bphi=-\int_\Omega
{\mathcal P}\nabla\cdot\bphi=0$, after integration by parts we get
\begin{equation}
\label{weak sense}
\begin{array}{l}
 \int_\Omega|D({\bu})+D({\ba})|^{p-2}(D({\bu})+
D({\ba})):D(\bphi)=\\
-\left({\bu}\cdot\nabla{\bu},\bphi\right)
-\left({\bu}\cdot\nabla{\ba},\bphi\right)
-\left({\ba}\cdot\nabla{\bu},\bphi\right)
-\left({\ba}\cdot\nabla{\ba},\bphi\right),
\end{array}
\end{equation}
for all $\bphi\in{\mathcal D}(\Omega)$, where for $n\times n$ matrices $A=(a_{ij}),\,B=(b_{ij})$ we use the
notation $A:B=\sum_{i,j=1}^n a_{ij}b_{ij}$.
Thus, the following definition for a {\em weak solution} to the problem \eqref{ls for u} is in order.

\begin{defi}
\label{weak solution} 
A vector field ${\bu}$ is said to be a {\em weak solution} to the problem \eqref{ls for u} if it has the following properties:
\begin{itemize}
\item[\ \ i)] ${\bu}\in W_{loc}^{1,p}(\overline{\Omega})$;
\item[\ ii)] ${\bu}$ satisfies \eqref{weak sense} for every
$\bphi\in{\mathcal D}(\Omega)$;
\item[iii)] ${\bu}$ satisfies \eqref{ls for u}$_2$-\eqref{ls for u}$_5$.
\end{itemize}
Similarly, a vector field $\bv$ is said to be a {\em weak solution} to the problem \eqref{lady-solo pb} if $\bv\in W_{loc}^{1,p}(\overline{\Omega})$
and satisfies \eqref{lady-solo pb}$_2$-\eqref{lady-solo pb}$_5$ and
\eqref{weak sense} with $\bu+\ba$ replaced by $\bv$, i.e.
$$
\int_\Omega|D(\bv)|^{p-2}D({\bv}):D(\bphi)
=-\left({\bv}\cdot\nabla{\bv},\bphi\right)
$$
for all $\bphi\in{\mathcal D}(\Omega)$.
\end{defi}

\begin{remark}
\label{derham} The use of divergence free test functions $\bphi$ in \eqref{weak sense}
eliminates the pressure $\mathcal P$, but it is a standard fact that it can be recovered due to
{\lq}De Rham's lemma{\rq} (cf. e.g. \cite[Lemma IV.1.1]{galdi}).
\end{remark}

We end this Section stating our main theorem, which we prove in Section \ref{main results}.

\begin{teo}
\label{main theorem} Let $p\ge2$.
Then, for any $\alpha\in\real$, problem \eqref{lady-solo pb} has a weak solution ${\bv}$, in the sense of Definition \ref{weak solution}.
\end{teo}


\section{Preliminary results}
\label{preliminary results}

$\indent$ In this Section we give some preliminary facts we shall need 
to prove our mains results in Section \ref{main results}. We begin with
Lemma \ref{ls lemma} below, which is due to Ladyzhenskaya and Solonnikov \cite[Lemma 2.3]{ls}. 
Our statement below differs slightly from \cite{ls} and, for convenience of the reader, we present its proof, which essentially can be found in \cite{ls} and \cite{fabio, fabio-tese}. 
\begin{lemma}
\label{ls lemma} Let $\Psi:\real\rightarrow\real$ be  
a strictly increasing function, $\delta$ be a number in the
interval $(0,1)$ and $t_0<T$.
 
{\em i)} If $z$ and $\varphi$ are differentiable functions in
the interval $[t_0,T]$ satisfying the inequalities
$$
\left\{\begin{array}{l}
z(t)\leq\Psi(z^\prime(t))+(1-\delta)\varphi(t),\\
\varphi(t)\geq\delta^{-1}\Psi(\varphi^\prime(t)),
\end{array}\right.
$$
for all $t\in[t_0,T]$, and $z(T)\leq\varphi(T)$, then
$$
z(t)\leq\varphi(t)\,,\;\forall\,t\in[t_0,T].
$$

{\em ii)} Suppose $\Psi(0)=0$.
If $z:[t_0,\infty)\to[0,\infty)$ is a non identically zero and non decreasing differentiable function, and satisfies the inequality $z(t)\leq\Psi\left(z^\prime(t)\right)$, for all $t\geq t_0$,
then $ \lim_{t\rightarrow\infty}z(t)=~\infty$. Besides, if
$\Psi(\tau)\leq c\tau^m$ for all $\tau\geq\tau_1$, for some constants $m>1,c>0$ and $\tau_1>0$, then 
$$
\liminf_{t\rightarrow\infty}t^{-\frac{m}{m-1}}z(t)>0\,;
$$
If, however, $\Psi(\tau)\leq c\tau$, for $\tau\geq\tau_1$, then
$$
\liminf_{t\rightarrow\infty}\mbox{e}^{-t/c}z(t)>0\,.
$$
\end{lemma}

\begin{proof}[\bf Proof] i) Suppose that $\varphi(t_1)<z(t_1)$ for some 
$t_1\in [t_0,T)$. Then, by the first inequality, we have $z(t_1)<\delta^{-1}\Psi(z'(t_1))$, and so, using the second inequality,
we have also  
$\delta^{-1}\Psi(\varphi'(t_1))\le\varphi(t_1)<z(t_1)<\delta^{-1}\Psi(z'(t_1))$, then, $\Psi(\varphi'(t_1))<\Psi(z'(t_1))$. 
Since $\Psi$ is strictly increasing,
it follows that $z'(t_1)>\varphi'(t_1)$. Consequently, $z(t)>\varphi(t)$
for all $t$ on a neighborhood on the right of $t_1$, and so, taking
$t_2$ to be the supremum of these points in $(t_1,T)$, we have $t_1<t_2<T$
and, by the previous reasoning, we have $z'(t)>\varphi'(t)$ for all
$t$ in $(t_1,t_2)$, but this yields a contradiction, since 
$z(t)-\varphi(t)>0$ is strictly positive at $t=t_1$ and must be zero at 
$t=t_2$. 

\smallskip

ii) Let $t_1\ge t_0$ such that $z(t_1)>0$ and $\lambda=\Psi^{-1}(z(t_1))$.
Notice that $\lambda>0$, since $\Psi(0)=0$ and $\Psi$ is strictly increasing. As $z$ is a nondecreasing function, we have that 
$z(t)\ge z(t_1)$ for all $t\ge t_1$. Then we claim that $z(t)\ge z(t_1)+\lambda (t-t_1)$ for all $t\ge t_1$. Indeed, the inequalities
$z(t)\ge z(t_1)$ and $z(t)\leq\Psi\left(z^\prime(t)\right)$ imply 
$z^\prime(t)\ge\Psi^{-1}(z(t))\ge\Psi^{-1}(z(t_1))=\lambda$. 
Thus, we have shown the first statement in part 2) of the Lemma. 
For the remainder, notice that, since $\lim_{t\to\infty}z(t)=\infty$,
there exists a $r$ such that $z(t)\ge\tau_1$
for all $t>r$, so from $\Psi(\tau)\le c\tau^m$ and $z(t)\le\Psi(z'(t))$ we have $z(t)\le c (z'(t))^m$ for all $t>r$, and the results then follow by
direct integrating this inequality.
\end{proof}

In the next lemma we collect three very useful inequalities.
The first can be found in many texts, as for instance in \cite{dibenedetto} and \cite[Lemma 2.1, p. 526]{barret-liu}. 
The third inequality contains Korn's
inequality 
(see \cite{neff}\footnote{In \cite{neff}, Korn's inequality is stated for dimension three. The result in dimension two can be obtained from the one in dimension three by extending the domain $U\subset\real^2$ to 
$U\times (0,1)$ and the vector field $\bv:U\to\real^2$ to 
$(\bv,0):U\times (0,1)\to\real^3$.}). 
The last one is a classical Poincar\'e type inequality; see e.g.
\cite[p.56]{galdi}. In these inequalities, $c_1,c_2$ 
are positive constants depending only on $p$ and, for the last two, on the domain $U$. 

\begin{lemma}
\label{inequalities}
\  
%
%

{\em \ \ i)}
$$
\langle|x|^{p-2}x-|y|^{p-2}y,x-y\rangle\geq
c_1|x-y|^2\left(|x|^{p-2}+|y|^{p-2}\right)\ge c_2|x-y|^p\, 
$$
for all $x,y\in\real^n$ and $p\ge2$.
%
%

{\em \ ii)}
$$
c_1|\mathbf{v}|_{1,p,U}\leq\parallel D(\mathbf{v})\parallel_{p,U}\leq
c_2|\mathbf{v}|_{1,p,U}\,,
$$
for all $\mathbf{v}\in{\mathcal D}^{1,p}(U)$ such that $\bv|\Gamma=0$.

{\em iii)}
$$
\parallel\mathbf{v}\parallel_{q,U}\leq
c_1\left(|\mathbf{v}|_{1,q,U}+\parallel\mathbf{v}\parallel_{1,\Gamma}\right),  
$$
for all $\mathbf{v} \in W^{1,q}(U)$.
In {\em ii)} and {\em iii)}, $U$ is an arbitrary bounded domain of $\mathbb{R}^n, \, n=2,3$, with a smooth boundary, $\Gamma$ is any
Lebesgue measurable subset of $\partial U$ with positive measure, and 
$1\leq p<\infty$.
\end{lemma}

\medskip

Next, we state a corollary of Brouwer fixed point theorem. 

\begin{lemma}
\label{brouwer} 
Let $F:\real^n\,\longrightarrow\,\real^n$ be a continuous map such that for some $\rho>0$, $F(\xi)\cdot\xi\geq0$ for all $\xi\in\real^n$ with
$|\xi|=\rho$. Then, there is a $\xi_0\in\real^n$ with $|\xi_0|\leq\rho$
such that $F(\xi_0)=0$.
\end{lemma}

For a proof, see \cite[Lemma VIII.3.1]{galdi} or 
\cite[p.\! 493]{evans}. 

\smallskip

The next lemma yields a solution $\mathbf{v}$ of the equation $\mbox{div}\,\mathbf{v}=f$ satisfying a nice estimate. This result is 
an important step in the proof of our main theorem.

\begin{lemma}
\label{div eq} Let $U$ be a locally Lipschtzian and bounded domain in $\real^n,\,n\geq 2$, and $1<q<\infty$. Then there is a constant $c$ such that, for any $f\in L^q(U)$ satisfying $\int_Uf=0$, there is a vector
field $\bv\in W_0^{1,q}(U)$ such that $\nabla\cdot\mathbf{v}=f$
and $\parallel\mathbf{v}\parallel_{1,q,U}\leq c\parallel f\parallel_{q,U}$.
\end{lemma}

See \cite[Teorema III.3.2]{galdi}. 

\medskip

The final result of this Section regards the {\em regularised distance} function to the 
boundary of a domain (an open connected set) in $\mathbb{R}^n$.

\begin{lemma}
\label{reg dist}
Let $V$ be a domain in $\real^n$ and
$\mbox{d}(x)=\mbox{dist}(x,\partial V)$, $x\in V$. 
Then, there is a function
$\rho\in C^\infty(V)$ such that for every $x\in V$ and any
derivative $\partial^\beta$, 
$\beta=(\beta_1,\cdots,\beta_n)\in{\mathbb Z}_{+}$, 
we have
\begin{equation}
\begin{array}{l}
\ \ \ \ \ \, \mbox{d}(x)\leq\rho(x) \ \ \ \ \mbox{and}\\
|\partial^\beta\rho(x)|\leq
k_\beta(\mbox{d}(x))^{1-|\beta|},
\end{array}\label{eq.2}
\end{equation}
where $k_{\beta}$ is a constant depending only on $\beta$ and $n$.
\end{lemma}

See 
\cite[Theorem VI.2]{stein}.






\section{Proof of Theorem \ref{main theorem} and other results}
\label{main results}

In this section we prove Lemma \ref{a} and our main theorem - Theorem \ref{main theorem}. Besides, we make some remarks, prove
a Proposition on the {\lq}uniform{\rq} distribution of 
energy dissipation (Proposition \ref{prop.2}) and a Theorem
regarding the uniqueness of solution of problem \eqref{lady-solo pb}.

\smallskip

We begin by proving Lemma \ref{a}.
As we observed in the Introduction, the proof of this lemma (the construction of $\ba$) is simpler in this paper (i.e. for the case $p>2$) than
for the classical one for newtonian fluids ($p=2$). For the construction
in the case $p=2$, see \cite[p.744]{ls} and references therein; see also
\cite[Lemma XI.7.1, p. 272]{galdi} and \cite[p. 46]{passerini}.


\begin{proof}[\bf Proof of Lemma \ref{a}.] Suppose we have a vector field
$\tilde\ba$ as in Lemma \ref{a}. Then the statements with respect to
$\ba=\alpha\tilde\ba$ follow, with $c$ depending on $p$, $\sup_{|x_1|>0}|\Sigma|$, $\sup_\Omega |\tilde\ba|$ and
$\sup_\Omega |\nabla\tilde\ba|$. Indeed, for property Lemma \ref{a}\,i), see 
\eqref{1 estimate}. For property ii), we have
$$
\begin{array}{c}
\int_{\Omega_{i,t-1,t}}|\nabla\ba|^p
\leq (\sup|\nabla\tilde\ba|^p)(\sup|\Sigma|){|\alpha|}^p
\end{array}
$$
and iii) follows from ii):
$$
\begin{array}{rl}
\int_{\Omega_t}|\nabla\ba|^p
&=\ \int_{\Omega_0}|\nabla\ba|^p
  +\sum_{i=1,2}\int_{\Omega_{i,t}}|\nabla\ba|^p\\
&\le |\Omega_0|(\sup|\nabla\tilde\ba|^p){|\alpha|}^p
  +\left( (\sup|\nabla\tilde\ba|^p)(\sup|\Sigma|)\right){|\alpha|}^pt.
\end{array}
$$
To construct a vector field
$\tilde\ba$ with the properties in the statement of Lemma \ref{a},
first we observe that it is enough to
construct in each outlet $\Omega_i$ a vector field $\mathbf{a}^i$
satisfying these properties in $\Omega_i$. Indeed, if we have this, then we
can obtain the desired vector field $\tilde\ba$ defined in $\Omega$ by
using appropriate cutoff functions. We omit this part of the proof and
refer to \cite[cap.VI]{galdi} for a similar
procedure in a domain with straight outlets and Poiseuille flows in place
of the vector fields ${\ba}^i$, to be constructed below.

We first construct $\tilde\ba$ in the case $n=2$.
By what we observed above, it is enough to
construct the vector field $\tilde\ba$ in an arbitrary outlet $\Omega_i$,
which we shall denote by $\Omega$ in this proof. 
Without loss of generality, we take 
$\Omega=\{x=(x_1,x_2)\in\real^2\, ;\, f_1(x_1)<x_2<f_2(x_1)\}$ for
smooth functions $f_1, f_2$ 
such that 
$f_1(x_1)\le-\frac{l_1}{2}$, \ $\frac{l_1}{2}\le f_2(x_1)$ 
and $f_2(x_1)-f_1(x_1)\le l_2$, for all $x_1\in\real$.  ($l_1<l_2$ are positive
numbers introduced in Section \ref{section 2}.) 
Then we set
%
$$
\tilde\ba=\nabla^\perp\zeta\equiv
(\partial_{x_2}\zeta,-\partial_{x_1}\zeta),
$$
for $\zeta(x_1,x_2)=\psi(x_2/\rho(x))$, where $\rho(x)$ is the regularised
distance to $\partial\Omega$ (see Lemma \ref{reg dist}) and $\psi:\real\to\real$ is a smooth
nondecreasing function such that $\psi(s)=0$ if $s<0$ and $1$, if $s>1$.
We notice that $\zeta$ is identically zero in the {\lq}lower band{\rq}
$\{x\in\Omega\, ;\, f_1(x_1)<x_2<0\}$ and identically one in a neighborhood of the {\lq}upper boundary{\rq} 
$\{x\in\partial\Omega\, ;\, x_2=f_2(x_1)\}$. In particular, $\tilde\ba$ is
a divergence free bounded vector field vanishing on a neighborhood of 
$\partial\Omega$ and  
%
$$
\int_{\Sigma}\tilde\ba\cdot\bn
=\int_{\Sigma}\zeta_{x_2}dx_2\\
=\zeta(x_1,f_2(x_1))-\zeta(x_1,f_1(x_1))=1\,.
$$
Now, because $\zeta$ is constant in a neighborhood of each of the two
components of  
$\partial\Omega$, we have that any derivative of $\zeta$ is zero
in this neighborhood and, thus, bounded in $\Omega$. 
Then $\tilde\ba$ and its derivatives are bonded function in $\Omega$.

\smallskip

In the case $n=3$, we take $\zeta(x_1,x')=\psi(|x'|/\rho(x))$,
$x'\equiv (x_2,x_3)\in\real^2$, where $\rho(x)$ is the regularised distance to $\partial\Omega$ (see Lemma \ref{reg dist}), $\psi$ is as above, but
$\psi(s)=0$ if $s<1$ and $1$, if $s>2$. Then we set
$$
\tilde\ba=\nabla\times(\zeta\bb)=(\nabla\zeta)\bb,
$$
where $\bb$ is the angle form in $\real^2$, i.e.  $\bb(x_2,x_3)=\frac{1}{2\pi(x_2^2+x_3^2)}(-x_3,x_2)$. Notice 
that $\zeta$ constant for $x'$ close to zero and equal to one in a neighborhood
of $\partial\Omega$ (i.e. $\rho(x)$ close to zero), and thus, $\zeta$
is a smooth function with bounded derivatives, vanishing in neighborhoods
of $x'=0$ and $\partial\Omega$. Therefore, $\tilde\ba$ is a smooth function vector with bounded derivatives. Beside, it is divergence free,
and, by Stokes theorem in the plane, we have
$\int_{\Sigma}\tilde\ba\cdot\bn=\int_{\partial\Sigma}\bb d\sigma=1$.
\end{proof}

\smallskip

To solve problem \eqref{ls for u}, first we shall solve the truncated
modified problem, $T>0$: 
\begin{equation}
\label{truncated}
\left\{\begin{array}{l}
\mbox{div}\{\big(\frac{1}{T}+|D(\bu)+D(\ba)|^{p-2}\big)(D(\bu)+D({\ba}))\}\\
={\bu}\cdot\nabla{\bu}+{\bu}\cdot\nabla{\ba}
+{\ba}\cdot\nabla{\bu}+{\ba}\cdot\nabla{\ba}+\nabla
{\mathcal P}\;\;\;\;\mbox{in}\;\Omega_T\\
\nabla\cdot{\bu}=0\;\;\;\;\mbox{in}\;\Omega_T\\
{\bu}=0\;\;\;\;\mbox{on}\;\partial\Omega_T\\
\end{array}\right. 
\end{equation}

Then we will use Lemma \ref{ls lemma} to obtain a weak solution
of \eqref{ls for u} by taking the limit, when $T\to\infty$, in the solution $\bu^T$ of \eqref{truncated}, extended by zero outside $\Omega_T$.

\begin{prop}
\label{prop.1} Let $p\ge2$ and $T>0$. Then problem \eqref{truncated} has a solution $(\bu^T,{\mathcal P})$ in 
${\mathcal D}_0^{1,p}(\Omega_T)\times L^p(\Omega_T)
\cap W^{2,l}(\Omega_t)\times W^{1,l}(\Omega_t)$, 
for any $t\in(0,T)$, where $l=2q/(p+q-2)$, being $q=2p+2$ if $n=3$
and any number in $[1,\infty)$ if $n=2$.  
\end{prop}

\begin{proof}[\bf Proof] 
The regularity part, i.e. $(\bu^T,{\mathcal P})\in  W^{2,l}(\Omega_t)\times W^{1,l}(\Omega_t)$, for any $t\in (0,T)$,
is a corollary of the proof of Theorem 1.2 in \cite{veiga1}. 
Notice that if $(\bu^T,{\mathcal P})$ is a weak solution with $\bu^T$ in ${\mathcal D}_0^{1,p}$ then $\bv=\bu^T+\ba$ is a weak solution in 
$W^{1,p}(\Omega_T)$ of
\begin{equation}
\label{truncated v}
\left\{\begin{array}{l}
\mbox{div}\{\big(\frac{1}{T}+|D(\bv)|^{p-2}\big)D(\bv)\} 
+{\bv}\cdot\nabla{\bv}+\nabla {\mathcal P}=0\ \ \ \mbox{in}\;\Omega_T\\
\nabla\cdot{\bv}=0\ \ \ \mbox{in}\;\Omega_T\\
{\bv}=\ba\ \ \ \mbox{on}\;\partial\Omega_T.
\end{array}\right.  
\end{equation}
The fact that we do not have here the homogeneous Dirichlet boundary
condition $\bv=0$ here in the whole boundary $\partial\Omega_T$ does not
affect the method given in \cite{veiga1} because 
$\ba=0$ in $\overline{(\partial\Omega_T)}\cap(\partial\Omega)$ and the remaing
part of $\partial\Omega_T$, i.e., $(\partial\Omega_T)/(\partial\Omega)$,
is interior to $\Omega_T$. 

Then we have only to show the existence of a weak solution for \eqref{truncated}. 
For simplicity, most of the time in this proof we shall write $\Omega_T=\Omega$ and $\bu^T=\bu$. 
Also we keep the notation $(\cdot,\cdot)$ with the integration over
$\Omega=\Omega^T$ in this proof, i.e. for (vector) functions $\bv,\bw$ such that $\bv\cdot\bw\in L^1(\Omega^T)$, $(\bv,\bw)=\int_{\Omega^T}\bv\cdot\bw$. 
We will apply the Galerkin method and the monotonicity method of
Browder-Minty (cf. \cite[Remark, p. 497]{evans}). The Browder-Minty method
is used due to the nonlinear term in the left hand side of \eqref{truncated}$_1$.

Let $\{\bphi^j;j=1,2,\cdots\}\subset{\mathcal D}(\Omega)$ be a denumerable 
and linearly independent set of functions whose linear hull is dense in
${\mathcal D}_0^{1,p}(\Omega)$. We shall write for $m=1,2,\ldots$,
\begin{equation}
{\bu}^m=\sum_{j=1}^mc_{jm}{\mathbb{\varphi}}^j\,, \label{eq.7}
\end{equation}
where $(c_{1m},\cdots,c_{mm})\in\real^m$ solves the algebraic system
\begin{equation}
\label{eq for um}
\begin{array}{c}
 \frac{1}{T}\int_\Omega (D({\bu}^m)+D({\ba})):D({\mathbb{\varphi}}^j)\\
+ \int_\Omega|D({\bu}^m)+D({\ba})|^{p-2}(D({\bu}^m)+
D({\ba})):D({\mathbb{\varphi}}^j)\\
+\left({\bu}^m\cdot\nabla{\bu}^m,{\mathbb{\varphi}}^j\right)
+\left({\bu}^m\cdot\nabla{\ba},{\mathbb{\varphi}}^j\right)
+\left({\ba}\cdot\nabla{\bu}^m,{\mathbb{\varphi}}^j\right)
+\left({\ba}\cdot\nabla{\ba},{\mathbb{\varphi}}^j\right)=0,
\end{array}
\end{equation}
$j=1,\ldots,m$.
To see that system \eqref{eq for um} has a solution $(c_{1m},\cdots,c_{mm})$, let 
${\mathbf F}=(F_1,\cdots,F_m):\real^m\to\real^m$ be the map such that,
for $\xi=(\xi_1,\ldots,\xi_m)\in\real^m$,  $F_j(\xi)$ is defined by
the left hand side of \eqref{eq for um} with ${\bu}^m=\sum_{j=1}^m\xi_j{\mathbb{\varphi}}^j$. 
By Lemma \ref{brouwer}, it is enough to show that there is a $\rho>0$
such that ${\mathbf{F}}(\xi)\cdot\xi\geq0$ for all $|\xi|=\rho$.
Since $\left({\bu}^m\cdot\nabla{\bu}^m,{\bu}^m\right)=
\left({\ba}\cdot\nabla{\bu}^m,{\bu}^m\right)=0$, we have
\begin{equation}
\label{fxixi}
\begin{array}{rl}
{\mathbf{F}}(\xi)\cdot\xi=& \frac{1}{T}\int_\Omega (D({\bu}^m)+D({\ba})):D(\bu^m)\\
&+ \int_\Omega|D({\bu}^m)+
D({\ba})|^{p-2}(D({\bu}^m)+D({\ba})):D({\bu}^m)\\
&+\left({\bu}^m\cdot\nabla{\ba},{\bu}^m\right)+
\left({\ba}\cdot\nabla{\ba},{\bu}^m\right).
\end{array}
\end{equation}
By lemmas \ref{a} and \ref{inequalities}, and H\"older and
Young inequalities, we obtain the following estimates (for some
small positive numbers $\varepsilon_i$ and some constants 
$C_{\varepsilon_i},c_{\cdot}$, which may depend on $m$):
\begin{equation}
\begin{array}{rl}
& \left|\left({\bu}^m\cdot\nabla{\ba},{\bu}^m\right)\right|=
 \left|\left({\bu}^m\cdot\nabla{\bu}^m,{\ba}\right)\right|\\
\leq &  \left|{\bu}^m\right|_{1,p}\left(\int_\Omega|{\ba}|^{p^\prime}
\left|{\bu}^m\right|^{p^\prime}\right)^{1/p^\prime}
\leq  c\,\left|{\bu}^m\right|_{1,p}^2\leq
\varepsilon_1|{\bu}^m|_{1,p}^p+C_{\varepsilon_1}; 
\end{array}
\label{eq.10}
\end{equation}
\begin{equation}
\label{eq.65}
\left|\left({\ba}\cdot\nabla{\ba},{\bu}^m\right)\right|\leq
\varepsilon_2|{\bu}^m|_{1,p}^p+C_{\varepsilon_2};
\end{equation}
$$
\begin{array}{c}
 \int_\Omega\left(\left|D({\bu}^m)+D({\ba})\right|^{p-2}\left(D({\bu}^m)+ D({\ba})\right)-|D({\ba})|^{p-2}D({\ba})\right):D({\bu}^m)\\
\geq  c_p\int_\Omega\left|D({\bu}^m) \right|^p\ge c_1|\bu^m|_{1,p}^p,
\end{array}
$$
then
\begin{equation}
\label{eq.66}
\begin{array}{rl}
& \int_\Omega\left|D({\bu}^m)+D({\ba})\right|^{p-2}\left(D({\bu}^m)+ D({\ba})\right):D({\bu}^m)\\
\geq&  c_1|{\bu}^m|_{1,p}^p+\int_\Omega|D({\ba})|^{p-2}D({\ba}):D({\bu}^m),\\
\geq&  
c_1|{\bu}^m|_{1,p}^p-\int_\Omega|D({\ba})|^{p-1}|D({\bu}^m)|\\
\geq& 
c_1|{\bu}^m|_{1,p}^p-\left(\int_\Omega|D({\ba})|^p\right)^{(p-1)/p}
\left(\int_\Omega|D({\bu}^m)|^p\right)^{1/p}\\
\geq& 
c_1|{\bu}^m|_{1,p}^p-\varepsilon_3|{\bu}^m|_{1,p}^p+C_{\varepsilon_3};
\end{array}
\end{equation}
\begin{equation}
\label{eq.67}
\begin{array}{rl}
& \int_\Omega \left(D(\bu^m)+D(\ba)\right):D(\bu^m)
= \int_\Omega |D(\bu^m)|^2+D(\ba):D(\bu^m)\\
\ge& \frac{1}{2}\int_\Omega |D(\bu^m)|^2 - \frac{1}{2}\int_\Omega |D(\ba)|^2
\ge  c_1|\bu^m|_{1,2}^2 - \frac{1}{2}\int_\Omega |D(\ba)|^2.
\end{array}
\end{equation}
Then, taking $\varepsilon_i,\;i=1,2,3,$ suficiently small, and noticing
that 
$|{\bu}^m|_{1,q}^q\ge c|{\bu}^m|_{1,2}^2\geq c_1|\xi|^2$ 
(notice that $|\xi|=|\bu^m|_{1,2}$ is a norm in $\real^m$), from \eqref{fxixi}-\eqref{eq.67} we get
$$
{\mathbf{F}}(\xi)\cdot\xi\geq c_1|\xi|^2-c_2\geq 0
$$
for all $\xi\in\real^m$ such that 
$|\xi|\geq\sqrt{c_2/c_1}$, for some positive constants $c_1,\,c_2$. 

\smallskip

Next, we notice that $|{\bu}^m|_{1,p}$ is uniformly bounded with respect to $m$. Indeed, multiplying $(\ref{eq for um})$ by $\xi_j$ and summing in
$j$ from $1$ to $m$, we obtain, as in \eqref{fxixi},
$$
\begin{array}{rl}
& \frac{1}{T}\int_\Omega |D(\bu^m)|^2 
          + \frac{1}{T} \int_\Omega D(\bu^m):D(\ba)\\
+& \int_\Omega|D({\bu}^m)+D({\ba})|^{p-2}(D({\bu}^m)+
D({\ba})):D({\bu}^m)\\
+&\left({\bu}^m\cdot\nabla{\ba},{\bu}^m\right)+
\left({\ba}\cdot\nabla{\ba},{\bu}^m\right)=0,
\end{array}
$$
and then, proceeding with similar estimates to obtain \eqref{eq.10},
\eqref{eq.66} and \eqref{eq.67}, we arrive at\footnote{Here we 
write explicitly $\Omega_T$, instead of $\Omega$, for future reference.} 
\begin{equation}
\label{y(T)}
\frac{1}{2T} \int_{\Omega_T}|D(\bu^m)|^2
+|{\bu}^m|_{1,p}^p\leq c,
\end{equation}
for some constant $c$.
Thus, there exists a subsequence of $\{\bu^m\}$, which we still shall denote by $\{\bu^m\}$, and a vector field ${\bu}\in{\mathcal D}_0^{1,p}(\Omega)$ such that$^\ast$
\begin{equation}
\label{convergence of um}
\begin{array}{l} {\bu}^m\rightharpoonup
{\bu}\;\;\;\mbox{in}\;{\mathcal D}_0^{1,p}(\Omega_T)\\
{\bu}^m\rightarrow {\bu}\;\;\;\mbox{in}\;L^q(\Omega_T)
\end{array}
\end{equation}
when $m\to\infty$, where $q\ge 1$ is any number less than the critical Sobolev exponent
$p^\ast:=\frac{np}{n-p}=\frac{3p}{3-p}$, if $n=3$ and $p<3$, and $1\le q<\infty$ is arbitrary, if
$p\ge n$ ($n=2,3$). In particular, $1\le q<\infty$ is arbitrary for $n=2$,
since $p>2$.

Now we want to pass to the limit in $(\ref{eq for um})$ when 
$m\to\infty$ and obtain it with $\bu$ in place of $\bu^m$ and with any ${\mathbb{\varphi}}\in {\mathcal D}(\Omega)$
in place of ${\mathbb{\varphi}}^j$. 
We begin by defining the operators
$$
A({\bw})=-\mbox{div}\{|D({\bw})+D({\ba})|^{p-2}(D({\bw})+ D({\ba}))\},
$$
and
$$
C(\bw)=- \frac{1}{T}\mbox{div}\{D(\bw)+D(\ba)\},
$$
for ${\bw}\in{\mathcal D}_0^{1,p}(\Omega)$. More precisely, 
$A$ and $C$ are operators from ${\mathcal D}_0^{1,p}(\Omega)$ into ${\mathcal D}_0^{1,p}(\Omega)^\prime$, defined by
$$
\langle A({\bw}),\bphi\rangle=
 \int_{\Omega_T}|D(\bw)+D(\ba)|^{p-2}(D(\bw)+D({\ba})):D(\bphi)
$$
and
$$
\langle C(\bw),\bphi\rangle
= \int_{\Omega_T}\left(D(\bw)+D(\ba)\right):D(\bphi).
$$
Notice that $D(\bw)+D(\ba)\in L^{p'}(\Omega)$ 
because $p>2\Rightarrow p'<p$ and $\Omega=\Omega_T$ is a bounded domain.
We also write
$$
B({\bw})=-\left({\bw}\cdot\nabla{\bw}+{\bw}\cdot\nabla{\ba}+
{\ba}\cdot\nabla{\bw}+{\ba}\cdot\nabla{\ba}\right).
$$
So we want to show that 
$\langle A(\bu)+C(\bu),{\bphi}\rangle=(B(\bu),\bphi)$ for every $\bphi$
in ${\mathcal D}(\Omega)$, or, equivalently, for every $\bphi$
in ${\mathcal D}_0^{1,p}(\Omega)$.

By \eqref{eq for um}, we have 
\begin{equation}
\label{aum+cum=bum}
\langle A(\bu^m)+C(\bu^m),\bphi \rangle = \left( B(\bu^m),\bphi\right)
\end{equation}
for all $\bphi\in {\mathcal D}_0^{1,p}(\Omega)$ and all $m=1,2,\cdots$. 

Since $|{\bu}^m|_{1,p}$ is uniformly bounded, by H\"older inequality we have that $\{A({\bu}^{m})\}$ is a bounded sequence in 
${\mathcal D}_0^{1,p}(\Omega)^{\prime}$, so there is a
${\mathbf\chi}\in{\mathcal D}_0^{1,p}(\Omega)^{\prime}$ and a further
subsequence $\{\bu^m\}$ such that
\begin{equation}
\label{aum}
\langle A({\bu}^{m}),{\bphi}\rangle\,\longrightarrow\,\langle
{\mathbf\chi},{\bphi}\rangle
\end{equation}
for all $\bphi\in{\mathcal D}_0^{1,p}(\Omega)$. 
Next, we show that
\begin{equation}
\label{bum}
\left(B({\bu}^{m}),\bphi\right)
\,\longrightarrow\,\left(B({\bu}),\bphi\right), \ \ 
\forall\ \bphi\in{\mathcal D}_0^{1,p}(\Omega).
\end{equation}
Fisrt we notice that
$$
\begin{array}{rl}
   &\left| \left( \bu^m\cdot\nabla\bu^m,\bphi\right)
          -\left( \bu\cdot\nabla\bu,\bphi\right) \right|\\
=  &\left| \left( (\bu^m-\bu)\cdot\nabla\bu^m,\bphi\right)
         +\left( \bu\cdot\nabla(\bu^m-\bu),\bphi\right) \right|\\
=  &\left| \left( (\bu^m-\bu)\cdot\nabla\bu^m,\bphi\right)
         -\left( \bu\cdot\nabla\bphi,\bu^m-\bu\right) \right|\\ 
\le&  \|\bu^m-\bu \|_q \| \nabla\bu^m \|_p \|\bphi\|_q 
 + \|\bu\|_q\|\nabla\bphi\|_p\|\bu^m-\bu\|_q
\longrightarrow0,
\end{array}
$$
where $q$ is large enough such that $\frac{2}{q}+\frac{1}{p}\le 1$
and less than $p^\ast:=\frac{np}{n-p}$ if $p<n$. Notice that
if $p<n$ then $n=3$ ($n=2,3$ in this paper and
$p>2$) and, since $p>2$,
we have $\frac{2}{p^\ast}+\frac{1}{p}<5/6$. 
Similarly, and more easily, we also have
$$
\left|\left(\bu^m\cdot\nabla\ba,\bphi\right)
- \left({\bu}\cdot\nabla{\ba},\bphi\right)\right|
\longrightarrow0
$$
and
$$
\left|\left(\ba\cdot\nabla\bu^m,\bphi\right)
- \left(\ba\cdot\nabla\bu,\bphi\right)\right|
\longrightarrow0.
$$
Thus we have shown \eqref{bum}. From \eqref{convergence of um} and the
fact that $p>2$ and $\Omega=\Omega_T$ is bounded, we also have 
$\lim\langle C(\bu^m),\bphi\rangle=\langle C(\bu),\bphi\rangle$
for all $\bphi\in{\mathcal D}_0^{1,p}(\Omega)$. Then,
from \eqref{aum+cum=bum}-\eqref{bum}, we have 
$\mathbf{\chi}+C(\bu)=B(\bu)$ in ${\mathcal D}_0^{1,p}(\Omega)^\prime$.
Then, to conclude the proof, it
remains to show that $\mathbf{\chi}=A(\bu)$. To see this, it is enough
now to show that $\langle A(\bu^m),\bu^m\rangle$ converges to 
$\langle {\mathbf\chi},\bu\rangle$, since, by Lemma \ref{inequalities}, the operator $A$ is monotone. Indeed, we have the following classical argument for monotone operators. From
$\langle A(\bu^m)-A(\bw),\bu^m-\bw\rangle\geq0$,
i.e.
$$
\langle A(\bu^m),\bu^m\rangle-\langle
A(\bu^m),\bw\rangle - \langle A(\bw),\bu^m\rangle+\langle A(\bw),\bw\rangle\geq0,
$$
if $\langle A(\bu^m),\bu^m\rangle$ converges to 
$\langle {\mathbf\chi},\bu\rangle$ then, by \eqref{convergence of um} and \eqref{aum}, we 
can take the limit in this inequality when $m\to\infty$ and obtain
$\langle\chi-A{\bw},{\bu}-{\bw}\rangle\geq0$,
for all $\bw\in{\mathcal D}_0^{1,p}(\Omega)$. Now replacing
${\bw}$ by ${\bu}-\lambda\bw$, for $\lambda\in\real^{+}$, we arrive at
$\langle\chi-A({\bu}-\lambda\bw),\bw\rangle\geq0$
for all $\bw\in{\mathcal D}_0^{1,p}(\Omega)$ and all $\lambda\in\real^{+}$.
Then the desired result follows, once one shows that 
$\lim_{\lambda\to0+}\langle A({\bu}-\lambda\bw),\bw \rangle 
= \langle A(\bu),\bw \rangle$. Here, we can show this using the 
Lebesgue's dominated convergence theorem, since the integrand
in $\langle A({\bu}-\lambda\bw),\bw \rangle$ is dominated, for any $\lambda\in(0,1)$, by some constant times the function
$(|D(\bu)|^{p-1}+|D(\bw)|^{p-1}+|D(\ba)|^{p-1})|D(\bw)|$, which belongs
to $L^1(\Omega)$.

To show that 
$\langle A(\bu^m),\bu^m\rangle=\left( B(\bu^m),\bu^m\right)
-\langle C(\bu^m),\bu^m\rangle$
converges to 
$\langle{\mathbf\chi},\bu\rangle$ which is equal to 
$\left( B(\bu),\bu \right)-\langle C(\bu),\bu\rangle$, we write
$$
\begin{array}{rl}
 &\left( B(\bu^m),\bu^m\right) - \left( B(\bu),\bu \right)\\
=&\left[\left( B(\bu^m),\bu^m\right) - \left( B(\bu),\bu^m \right)\right]
   + \left( B(\bu),\bu^m - \bu \right)
\end{array}
$$
and notice that the two last terms converge to zero, when $m\to\infty$,
by the estimates above we used to obtain \eqref{bum}.  
It is easy to see, using again \eqref{convergence of um} and the fact that
$p>2$ and $\Omega=\Omega_T$ is bounded, that we have also 
$\lim\langle C(\bu^m),\bu^m\rangle=\langle C(\bu),\bu\rangle$.
\end{proof}


Next, given any $t>0$, we show that the solution $\bu^T$ of \eqref{truncated}
is uniformly bounded in ${\mathcal D}_0^{1,p}(\Omega_t)$, with respect to
$T$, for $T\ge t+1$. Proceeding similarly to \cite{ls}, we introduce the function 
\begin{equation}
\label{def of y}
y(t)=\frac{1}{T}\left|{\bu^T}\right|_{1,2,\Omega_t}^2+
\left|{\bu^T}\right|_{1,p,\Omega_t}^p, \ \ t>0, \ T\ge t+1. 
\end{equation}

In the sequel we write $\bu^T=\bu$ and often $\bu+\ba=\bv$. 
Multiplyng \eqref{truncated}$_1$ by $\bu$ and integrating by parts, using that $\bu|\partial\Omega=~0$, we have
\begin{equation}
\label{integrating by parts}
\begin{array}{rl}
&\frac{1}{T}\int_{\Omega_t}\left|D({\bu})\right|^2
   +\int_{\Omega_t}\left|D(\bv)\right|^{p-2}D(\bv):D({\bu})\\
=& -\frac{1}{T}\int_{\Omega_t}D({\ba}):D({\bu})
   +\int_{\Omega_t}\big(
\bu\cdot\nabla{\bu}\cdot{\ba}-{\ba}\cdot\nabla{\bu}\cdot{\bu}-
{\ba}\cdot\nabla{\ba}\cdot{\bu}\big)\\
& +\int_{\Gamma_t}\big( \bu\cdot\left(D(\bv)\bn\right)
    +\bu\cdot(|D(\bv)|^{p-2}D(\bv)\bn) \big)\\
&-\int_{\Gamma_t}\big(   
     \frac{1}{2}|\bu|^2(\bu\cdot\bn)+(\bu\cdot\ba)(\bu\cdot\bn)
      +(\bu\cdot\bn){\mathcal P}\big),
\end{array}
\end{equation}
where $\Gamma_t=\Sigma(t)\cup\Sigma(-t)$. 
First we estimate the {\lq}interior{\rq} integrals $\int_{\Omega_t}\cdots$. Using Young inequality and Lemma \ref{a}, we get
\begin{equation}
\label{dadu}
\begin{array}{c}
-\int_{\Omega_t}D(\vec{\mbox{a}}):D({\bu})\leq
\varepsilon\int_{\Omega_t}\left|D({\bu})\right|^2+C_\varepsilon,
\end{array}
\end{equation}
and
\begin{equation}
\label{ua}
\begin{array}{rl}
&\left|\int_{\Omega_t}\big(
\bu\cdot\nabla{\bu}\cdot{\ba}-{\ba}\cdot\nabla{\bu}\cdot{\bu}-
{\ba}\cdot\nabla{\ba}\cdot{\bu}\big)\right|\\
\le & |\bu|_{1,p,\Omega_t}
\big(\int_{\Omega_t}|\ba|^{p\prime}|\bu|^{p'}\big)^{1/p'}
+ \left|\int_{\Omega_t}\big({\ba}\cdot\nabla{\bu}\cdot{\bu}-
{\ba}\cdot\nabla{\ba}\cdot{\bu}\big)\right|\\
\le & \varepsilon |\bu|_{1,p,\Omega_t}^p+C_\varepsilon t,
\end{array}
\end{equation}
%
where $\varepsilon>0$ is fixed below.
Besides, proceeding as in \eqref{eq.66}, we get
\begin{equation}
\label{dvdvdu}
\begin{array}{l}
 \int_{\Omega_t}\left|D(\bv)\right|^{p-2}D(\bv):D({\bu})\\
\geq c_p\left|{\bu}\right|_{1,p,\Omega_t}^p+ \int_{\Omega_t}\left|D({\ba})
\right|^{p-2}D({\ba}):D({\bu})
\end{array}
\end{equation}
and
\begin{equation}
\label{dadadu}
\begin{array}{c}
\left|\int_{\Omega_t}\left|D({\ba})\right|^{p-2}D({\ba}):D({\bu})\right|
\leq\varepsilon\left|{\bu}\right|_{1,p,\Omega_t}^p
+C_\varepsilon t.
\end{array}
\end{equation}
Then, from \eqref{integrating by parts}-\eqref{dvdvdu} and taking $\varepsilon\ll 1$, we obtain
\begin{equation}
\label{first estimate for y}
y(t)\leq c_1t+I, 
\end{equation}
where
\begin{equation}
\label{def of I}
\begin{array}{c}
I=\int_{\Gamma_t}\big[\frac{1}{T}\bu\cdot D(\bv)\bn
   +\bu\cdot |D(\bv)|^{p-2}D(\bv)\bn
   -\frac{1}{2}|\bu|^2(\bu\cdot\bn)\\
   -(\bu\cdot\ba)(\bu\cdot\bn)-(\bu\cdot\bn){\mathcal P}\big].
\end{array}
\end{equation}

Now the idea is to control the boundary integral $I$ by the interior
integral $y(t)$, but if for instance one tries to apply the trace theorem then
higher order derivatives arise.  To achieve that purpose
we use the clever idea given in \cite{ls} for the case $p=2$, that is, to
integrate $I\equiv I(t)$ from $\eta-1$ to $\eta$, for $\eta>1$, or better,
integrate the estimate \eqref{first estimate for y}.
Thus we introduce the function
\begin{equation}
\label{def of z}
z(\eta)=\int_{\eta-1}^\eta y(t)dt.
\end{equation}
Notice that since $y$ is a nondecreasing function we have 
$y(\eta-1)\le z(\eta)\le y(\eta)$ for all $\eta>1$, thus estimating $y$ is the same as estimating $z$. Another interesting feature of the function
$z$ is that
\begin{equation}
\label{derivative of z}
z'(\eta)=y(\eta)-y(\eta-1)=\frac{1}{T}|\bu|_{1,2,\Omega_{\eta-1,\eta}}^2
+ |\bu|_{1,p,\Omega_{\eta-1,\eta}}^p.
\end{equation}
Then if we estimate $\int_{\eta-1}^\eta I(t)dt$ in terms of
$|\bu|_{1,p,\Omega_{\eta-1,\eta}}^p$ and
$\frac{1}{T}|\bu|_{1,2,\Omega_{\eta-1,\eta}}^2$, in the end, in virtue of
\eqref{first estimate for y}, we shall obtain a estimate for $z(\eta)$ in terms of $z'(\eta)$. Then we shall use Lemma \ref{ls lemma} to get the desired estimate for $z(\eta)$. Let's do the details.

By \eqref{first estimate for y} and \eqref{def of z}, and the fact that
$\int_{\eta-1}^\eta\int_{\Gamma_t}\cdot
=\int_{\Omega_{\eta-1,\eta}}\cdot$, we have
\begin{equation}
\label{first estimate for z}
z(\eta)\equiv\int_{\eta-1}^\eta y(t)\,dt\leq
c_1\eta+\frac{1}{T}I_1+I_2+I_3+I_4+I_5,
\end{equation}
where
$$
\begin{array}{rll}
I_1&=& \int_{\Omega_{\eta-1,\eta}}\bu\cdot D(\bv)\bn\\
I_2&=& \int_{\Omega_{\eta-1,\eta}}\bu\cdot |D(\bv)|^{p-2}D(\bv)\bn\\
I_3&=&-\int_{\Omega_{\eta-1,\eta}}\frac{1}{2}|\bu|^2(\bu\cdot\bn)\\
I_4&=&-\int_{\Omega_{\eta-1,\eta}}(\bu\cdot\ba)(\bu\cdot\bn)\\
I_5&=&-\int_{\Omega_{\eta-1,\eta}}{\mathcal P}(\bu\cdot\bn).
\end{array}
$$
Using H\"older inequality, Lemma \ref{a}\,ii),
Poincar\'e inequality (Lemma \ref{inequalities}\,iii)) and Young inequality, we have 
$$
\begin{array}{rll}
|I_2|&\leq&\int_{\Omega_{\eta-1,\eta}} |D(\bv)|^{p-1}|\bu|\\
     &\leq& c\left(\int_{\Omega_{\eta-1,\eta}}|D(\bu)|^{p-1}|\bu|
             +\int_{\Omega_{\eta-1,\eta}}|D(\ba)|^{p-1}|\bu|\right)\\
           &\leq& c \left(|\bu|_{1,p,\Omega_{\eta-1,\eta}}^{p-1}
\|\bu\|_{p,\Omega_{\eta-1,\eta}}+
|\ba|_{1,p,\Omega_{\eta-1,\eta}}^{p-1}
\|\bu\|_{p,\Omega_{\eta-1,\eta}}\right)\\
&\leq& c\left(|\bu|_{1,p,\Omega_{\eta-1,\eta}}^p+|\bu|_{1,p,\Omega_{\eta-1,\eta}}\right),
\end{array}
$$
so
\begin{equation}
\label{I2}
|I_2|\leq
c\left(z^\prime(\eta)+z^\prime(\eta)^{1/p}\right).
\end{equation}
Analogously,
\begin{equation}
\label{I1}
|I_1|\leq
c\left(z^\prime(\eta)+z^\prime(\eta)^{1/2}\right).
\end{equation}
Regarding $I_3$ and $I_4$, using Sobolev embedding, we get
\begin{equation}
\label{I3 I4}
\begin{array}{rl} |I_3|+|I_4| \ \leq&
\int_{\Omega_{\eta-1,\eta}}\frac{1}{2}|\bu|^3
    + \int_{\Omega_{\eta-1,\eta}} c|\bu|^2\\
 \ \leq & 
c\big(|\bu|_{1,p,\Omega_{\eta-1,\eta}}^3+|\bu|_{1,p,\Omega_{\eta-1,\eta}}^2
\big)\\
\ = & c\left(z^\prime(\eta)^{3/p}+z^\prime(\eta)^{2/p}\right).
\end{array}
\end{equation}
To estimate $I_5$, we use Lemma \ref{div eq}. Let $\bw$ be a vector
field in $W_0^{1,p}(\Omega_{\eta-1,\eta})$ such that
$\nabla\cdot\bw=\bu\cdot\bn$ and 
$|\bw|_{1,p,\Omega_{\eta-1,\eta}}\le c|\bu|_{p,\Omega_{\eta-1,\eta}}$,
where $c$ is some constant, independent of $\bw$ and $\bu$.
Then, using the equation \eqref{truncated}$_1$, we can write 
$$
\begin{array}{rl}
I_5\ =& -\int_{\Omega_{\eta-1,\eta}}{\mathcal P}(\bu\cdot\bn)=\int_{\Omega_{\eta-1,\eta}}\nabla{\mathcal P}\bw\\ 
\ =& \int_{\Omega_{\eta-1,\eta}}|D(\bv)|^{p-2}D(\bv):D(\bw)
      + \int_{\Omega_{\eta-1,\eta}}(\bu\cdot\nabla\bu)\cdot\bw\\
   &  + \int_{\Omega_{\eta-1,\eta}}(\bu\cdot\nabla\ba)\cdot\bw
      + \int_{\Omega_{\eta-1,\eta}}(\ba\cdot\nabla\bu)\cdot\bw
      + \int_{\Omega_{\eta-1,\eta}}(\ba\cdot\nabla\ba)\cdot\bw.
\end{array}
$$
Thus, proceeding with similar estimates to those used to obtain
\eqref{I1}-\eqref{I3 I4}, we arrive at  
\begin{equation}
\label{I5}
|I_5|\leq
c\left(z^\prime(\eta)+z^\prime(\eta)^{1/p}+z^\prime(\eta)^{2/p}
+z^\prime(\eta)^{3/p}\right).
\end{equation}
From \eqref{first estimate for z}-\eqref{I5}, we have
\begin{equation}
\label{second estimate for z}
z(\eta)\leq c_1\eta+\Psi\left(z'(\eta)\right),
\end{equation}
for all $\eta\ge1$, with 
$\Psi(\tau)=c_2\left(\tau+\tau^{1/2}+\tau^{1/p}+\tau^{2/p}+\tau^{3/p}\right)$. 
Now, from \eqref{y(T)} and the weak convergence 
\eqref{convergence of um}$_1$, we have $y(T)\le c$ for
some constant $c$ (independent of $T$), so by $z(T)\le y(T)$ and by assuming that $c_1\ge c$,
without loss of generality, we have 
\begin{equation}
\label{choice of c1}
z(T)\le c_1 T. 
\end{equation}
Next, let $c_3>0$ be a constant such that
\begin{equation}
\label{choice of c3}
2c_1+c_3\ge 2\Psi(2c_1). 
\end{equation}
Then, by \eqref{second estimate for z}-\eqref{choice of c3}, we have
the conditions of Lemma \ref{ls lemma}\,i) satisfied, with
$\varphi(\eta)=2c_1\eta+c_3$, $\delta=1/2$, $t_0=1$ and $T>1$ (arbitrary). Therefore,
$z(\eta)\le 2c_1\eta+c_3$ for all $\eta\ge1$, and hence, since 
$y(\eta-1)\le z(\eta)$, we conclude that
there are (new) constants $c_1,c_2$ such that
\begin{equation}
\label{estimate for y}
y(t):=\frac{1}{T}|\bu^T|_{1,2,\Omega_t}^2 + |{\bu^T}|_{1,p,\Omega_t}^p
\le c_1 t +c_2, 
\end{equation}
for all $t>0$ and $T\ge t+1$.

Having the estimate \eqref{estimate for y}, we complete now the proof
of our main result.

\begin{proof}[\bf Proof of Theorem \ref{main theorem}.] 
Let $\bu^k$ be the solution of \eqref{truncated} in $\Omega_k$, $k=3,4,\cdots$, whose existence is assured by Proposition \ref{prop.1}, and set 
$\bu^k=0$ in $\Omega/\Omega_k$. By \eqref{estimate for y}, 
for each $j=2,3,\cdots$, the sequence $\{\bu^k\}_{k\ge j+1}$ is weakly compact in $W^{1,p}(\Omega_j)$, thus, by a diagonalization process
we obtain a subsequence, which we also denote by $\{\bu^k\}$, and an
$\bu$ in $W_{\mbox{\tiny loc}}^{1,p}({\overline\Omega})$ such that
\begin{equation}
\label{convergence of uk}
\begin{array}{l}
\bu^k\;\rightharpoonup\;\bu\;\;\;\;\;\mbox{in}\;\;W^{1,p}(\Omega_t)\\
\bu^k\;\rightarrow\;\bu\;\;\;\;\;\mbox{in}\;\;L^q(\Omega_t),
\end{array}
\end{equation}
for any $t>0$, where $q\ge 1$ is arbitrary, if $p\ge n$, and less than $p^\ast:=\frac{3p}{3-p}$, if $n=3$ and $p<3$. 
(Cf. \eqref{convergence of um}). 
Besides, by \eqref{convergence of uk}$_1$, the estimate 
\eqref{estimate for y} and the fact that 
$\bu^k\in{\mathcal D}_0^{1,p}(\Omega)$, we have that the limit $\bu$ satisfies 
\eqref{ls for u}$_2$-\eqref{ls for u}$_5$. Then, to conclude the proof
of Theorem \ref{main theorem}, it remains to prove that $\bu$ satisfies
the equation \eqref{ls for u}$_1$, in the weak sense \eqref{weak sense}.
Again, we shall use the Browder-Minty method, due to the shear dependent
viscosity. The idea here is to mimic the proof of Proposition \ref{prop.1},
paying attention that now $\Omega$ is not a bounded domain and
$D(\bu)$ is only locally integrable in $\overline{\Omega}$. This lead
us to localize the arguments and operators used in that proof, as follows.

Given $\bphi\in{\mathcal D}(\Omega)$, letting $k_0\in\N$ such that
$\mbox{supp}\,\bphi\subset\Omega_{k_0-1}$, we have 
\begin{equation}
\label{eq for uk}
\int_{\Omega_{k_0}}{\mathbb S}_k(\bu^k):D(\bphi)
=\int_{\Omega_{k_0}}B(\bu^k)\cdot\bphi,
\end{equation}
for all $k\ge k_0$, where
$${\mathbb S}_k(\bw)=\left(\frac{1}{k}+\left|D(\bw)+D(\ba)\right|^{p-2}\right)\left(D(\bw)+ D(\ba)\right)$$ and
$$B(\bw)=-\left(\bw\cdot\nabla\bw+\bw\cdot\nabla\ba
+\ba\cdot\nabla\bw+\ba\cdot\nabla\ba\right).$$ 
Then, we want to pass to the limit in \eqref{eq for uk} when $k\to\infty$
and obtain \eqref{weak sense}.
Let $\zeta:\Omega\longrightarrow\real_+$ be a smooth function such that $\zeta=1$ in $\mbox{supp}\,\bphi$ and
$\zeta=0$ in $\Omega\setminus\Omega_{k_0}$ and $A_\zeta,\;A_{\zeta,k}$ be the operators defined by
$$
\begin{array}{c}
  \langle A_{\zeta,k}\bw_1,\bw_2\rangle
= \int_{\Omega_{k_0}}{\mathbb S}_k(\bw_1):D(\bw_2)\zeta\\
  \langle A_\zeta\bw_1,\bw_2\rangle
= \int_{\Omega_{k_0}}{\mathbb S}(\bw_1):D(\bw_2)\zeta,
\end{array}
$$
on the space 
$$V_0\equiv W^{1,p}(\Omega_{k_0},\partial\Omega\cap\partial\Omega_{k_0})
:=\left\{\bw\in W^{1,p}(\Omega_{k_0})\,;\bw=0\;\mbox{in}\;\partial\Omega\cap\partial\Omega_{k_0}\right\},$$ 
where 
$${\mathbb S}(\bw)=|D(\bw)+D(\ba)|^{p-2}(D(\bw)+ D(\ba)).$$
Thus, \eqref{eq for uk} becomes
\begin{equation}
\label{auk=buk}
\langle A_{\zeta,k}(\bu^k),\bphi \rangle = \left( B(\bu^k),\bphi\right)
\end{equation}
and \eqref{weak sense} becomes
\begin{equation}
\langle A_{\zeta}(\bu),\bphi \rangle = \left( B(\bu),\bphi\right).
\end{equation}
We notice, as $\zeta$ is a nonnegative function, that $A_{\zeta,k}$ is still
a monotone operator. Besides, $\{A_{\zeta,k}(\bu^k)\}$ is a bounded
sequence in $V_0^\ast$, then, up to a subsequence, we have
$A_{\zeta,k}(\bu^k)\;\sta{\ast}{\rightharpoonup}\;\chi_\zeta$
for some $\chi_\zeta$ in $V_0^\ast$. As in \eqref{bum}, we also have 
\begin{equation}
\label{buk}
\left(B(\bu^k),\bphi\right)\,\longrightarrow\,\left(B({\bu}),\bphi\right). \end{equation}
Then, by \eqref{auk=buk}, we obtain 
$\langle \chi_\zeta,\bphi \rangle = \left( B(\bu),\bphi\right)$, so it
remains to show that $\chi_\zeta=A_{\zeta}(\bu)$. To obtain this, from
the monotonicity of $A_{\zeta,k}$, it is enough to prove that
$\langle A_{\zeta,k}(\bu^k),\bu^k \rangle$ converges to 
$\langle \chi_\zeta,\bu\rangle$. Indeed,
\begin{equation}
\langle A_{\zeta,k}\bu^k,\bu^k\rangle -
\langle A_{\zeta,k}\bu^k,\bw\rangle-
\langle A_{\zeta,k}\bw,\bu^k\rangle+
\langle A_{\zeta,k}\bw,\bw\rangle\geq0,
\end{equation}
for all $\bw\in V_0$ and, by \eqref{convergence of uk}, 
$\langle A_{\zeta,k}\bw,\bu^k\rangle$ and 
$\langle A_{\zeta,k}\bw,\bw\rangle$ tend, respectively, to
$\langle A_{\zeta}\bw,\bu\rangle$ and 
$\langle A_{\zeta}\bw,\bw\rangle$, when $k\to\infty$. Then, once
we have $\lim_{k\to\infty}\langle A_{\zeta,k}(\bu^k),\bu^k \rangle
=\langle \chi_\zeta,\bu\rangle$, we shall have 
$\langle\chi_\zeta-A_\zeta({\bu}-\lambda\bw),\bw\rangle\geq0$ for all
$\bw\in V_0$ and all $\lambda\ge0$, and by Lebesgue's dominated
convergence theorem, 
$\lim_{\lambda\to0+}\langle A_\zeta({\bu}-\lambda\bw),\bw\rangle=
\langle A_\zeta({\bu}),\bw\rangle$, hence $\chi_\zeta=A_\zeta(\bu)$.
Let us show then that 
$\lim_{k\to\infty}\langle A_{\zeta,k}(\bu^k),\bu^k \rangle
=\langle \chi_\zeta,\bu\rangle$. We compute $\langle \chi_\zeta,\bu\rangle$
and
$\lim_{k\to\infty}\langle A_{\zeta,k}(\bu^k),\bu^k \rangle$ using
directly the equation \eqref{truncated}$_1$, with $T=k$. Multiplying
this equation by $\zeta\bu$ and integrating by parts in $\Omega_{k_0}$,
we arrive at 
\begin{equation}
\label{auku}
\begin{array}{rl}
&\langle A_{\zeta,k}\bu^k,\bu\rangle\\
=&\int_{\Omega_{k_0}}B(\bu^k)\cdot\zeta\bu-
\int_{\Omega_{k_0}}\mathcal{P}^k\bu\cdot\nabla\zeta-\frac{1}{2}\int_{\Omega_{k_0}}
\bu\cdot {\mathbb S}(\bu^k)\cdot\nabla\zeta\\
&-\frac{1}{k}\left(\int_{\Omega_{k_0}}\bu\cdot
D(\bu^k)\cdot\nabla\zeta+\int_{\Omega_{k_0}}\bu\cdot
D(\ba)\cdot\nabla\zeta\right),
\end{array}
\end{equation}
where $\mathcal{P}^k$ is the pressure function associated with $\bu^k$.
From \eqref{estimate for y}, we have
\begin{equation}
\label{last term in auku}
\frac{1}{k}\left(\int_{\Omega_{k_0}}\bu\cdot
D(\bu^k)\cdot\nabla\zeta+\int_{\Omega_{k_0}}\bu\cdot
D(\ba)\cdot\nabla\zeta\right)\leq
\frac{c_{k_0}}{k}\;\longrightarrow\;0.
\end{equation}
and that $\{\mathbb{S}(\bu^k)\}$ is uniformly bounded in $L^{p'}(\Omega_{k_0})$, so there is a 
$\chi_{p^\prime}\in L^{p^\prime}(\Omega_{k_0})$ such that
\begin{equation}
\label{fourth term in auku}
\lim_{k\to\infty}\int_{\Omega_{k_0}}\bu\cdot
{\mathbb S}(\bu^k)\cdot\nabla\zeta=\int_{\Omega_{k_0}}
\bu\cdot \chi_{p^\prime}\cdot\nabla\zeta\,.
\end{equation}
Similarly to the proof of \eqref{bum}, we also have
\begin{equation}
\label{buk}
\lim_{k\to\infty}\int_{\Omega_{k_0}}B(\bu^k)\cdot\zeta\bu
=\int_{\Omega_{k_0}}B(\bu)\cdot\zeta\bu.
\end{equation}

Next, we show that
\begin{equation}
\label{third term in auku}
\int_{\Omega_{k_0}}\mathcal{P}^k\bu\cdot\nabla\zeta
\longrightarrow\int_{\Omega_{k_0}}\mathcal{P}\,\bu\cdot\nabla\zeta,
\end{equation}
for some further subsequence of $k\to\infty$,
where, up to a constant, $\mathcal{P}$ is the pressure function associated
with $\bu$. For this, it is enough to show that there is a  $\mathcal{P}\,\in L^{p^\prime}(\Omega_{k_0})$ ($p'=p/(p-1)$) such that
$\mathcal{P}^k\rightharpoonup\mathcal{P}\;\;\mbox{in}\;\;L^{p^\prime}(\Omega_{k_0})$, i.e. $\{{\mathcal P}^k\}$ is uniformly bounded in
$L^{p^\prime}(\Omega_{k_0})$. Let us assume, without loss of generality, 
$\int_{\Omega_{k_0}}\mathcal{P}^kdx=0$. 
Writing
$$g=\left|\mathcal{P}^k\right|^{p^\prime-2}\mathcal{P}^k-|\Omega_{k_0}|^{-1}\int_{\Omega_{k_0}}
\left|\mathcal{P}^k\right|^{p^\prime-2}\mathcal{P}^kdx,
$$
by Lemma \ref{div eq} there exist a constant $c$ (independent of $k$)
and a vector field $\bpsi\in W_0^{1,p}(\Omega_{k_0})$ such that
\begin{equation}
\label{eq.311}
\left\{\begin{array}{l}
\nabla\cdot\bpsi=g\\
\left\|\bpsi\right\|_{1,p,\Omega_{k_0}}\leq
c\left\|\mathcal{P}^k\right\|_{p^\prime,\Omega_{k_0}}^{\frac{1}{p-1}}.
\end{array}\right.
\end{equation}
Notice that $\int_{\Omega_{k_0}}gdx=0$, $g\in L^p(\Omega_{k_0})$ and
$\| g \|_{p,\Omega_{k_0}}\leq 
2\| \mathcal{P}^k \|_{p^\prime,\Omega_{k_0}}^{\frac{1}{p-1}}$.
Then,
\begin{equation}
\begin{array}{rl}
\int_{\Omega_{k_0}}\left|\mathcal{P}^k\right|^{p^\prime}=&\int_{\Omega_{k_0}}\left(\left|\mathcal{P}^k\right|^{p^\prime-2}\mathcal{P}^k\right)\mathcal{P}^k\\
=&\int_{\Omega_{k_0}}g\mathcal{P}^kdx+|\Omega_{k_0}|^{-1}\left(\int_{\Omega_{k_0}}
\left|\mathcal{P}^k\right|^{p^\prime-2}\mathcal{P}^k\right)\int_{\Omega_{k_0}}
\mathcal{P}^kdx\\
=&\int_{\Omega_{k_0}}
\mathcal{P}^k\nabla\cdot\bpsi=\int_{\Omega_{k_0}}S_k(\bu^k):D(\bpsi)+
\int_{\Omega_{k_0}}B(\bu^k)\cdot\bpsi,
\end{array}
\end{equation}
where, for the last iguality, we used equation \eqref{truncated}$_1$.
Using again \eqref{eq.311} and previous estimates, it follows that
\begin{equation}
\begin{array}{rl}
&\int_{\Omega_{k_0}}S_k(\bu^k):D(\bpsi)+
\int_{\Omega_{k_0}}B(\bu^k)\cdot\bpsi\\
\leq &
c(\left\|\bu^k\right\|_{1,p,\Omega_{k_0}}+\left\|\ba\right\|_{1,p,
\Omega_{k_0}})\left\|\mathcal{P}^k\right\|_{p^\prime,\Omega_{k_0}}^{\frac{1}{p-1}}.
\end{array}
\end{equation}
Therefore,
$$\left\|\mathcal{P}^k\right\|_{p^\prime,\Omega_{k_0}}\leq c\left(\left\|\bu^k
\right\|_{1,p,\Omega_{k_0}}+\left\|\ba\right\|_{1,p,\Omega_{k_0}}\right)\leq
C,
$$
as we wished.

From \eqref{auku}-\eqref{third term in auku}, we obtain
\begin{equation}
\label{chi zeta u}
\begin{array}{rl}
 &\langle\chi_\zeta,\bu\rangle\\
=&\int_{\Omega_{k_0}}B(\bu)\cdot(\zeta\bu)-
\int_{\Omega_{k_0}}\mathcal{P}\,\bu\cdot\nabla\zeta-\frac{1}{2}\int_{\Omega_{k_0}}
\bu\cdot \chi_{p^\prime}\cdot\nabla\zeta\,.
\end{array}
\end{equation}

Now, replacing $\bu$ by $\bu^k$ in \eqref{auku}, we have
\begin{equation}
\label{aukuk}
\begin{array}{rl}
&\langle A_{\zeta,k}\bu^k,\bu^k\rangle\\
=&\int_{\Omega_{k_0}}B(\bu^k)\cdot (\zeta\bu^k)-\int_{\Omega_{k_0}}\mathcal{P}^k\bu^k\cdot\nabla\zeta-\frac{1}{2}
\int_{\Omega_{k_0}}\bu^k\cdot S(\bu^k)\cdot\nabla\zeta\\
&-\frac{1}{k}\left(\int_{\Omega_{k_0}}\bu\cdot
D(\bu^k)\cdot\nabla\zeta+ \int_{\Omega_{k_0}}\bu\cdot
D(\ba)\cdot\nabla\zeta\right),
\end{array}
\end{equation}
and taking the limit when $k\to\infty$ in the right hand side here, 
analogously to the steps we did to obtain \eqref{chi zeta u}, we
get the right hand side of \eqref{chi zeta u}, i.e.
$$
\begin{array}{rl}
&\lim_{k\to\infty}\left\{\int_{\Omega_{k_0}}B(\bu^k)\cdot (\zeta\bu^k)-\int_{\Omega_{k_0}}\mathcal{P}^k\bu^k\cdot\nabla\zeta-\frac{1}{2}
\int_{\Omega_{k_0}}\bu^k\cdot S(\bu^k)\cdot\nabla\zeta\right.\\
&\left.-\frac{1}{k}\left(\int_{\Omega_{k_0}}\bu\cdot
D(\bu^k)\cdot\nabla\zeta+ \int_{\Omega_{k_0}}\bu\cdot
D(\ba)\cdot\nabla\zeta\right)\right\}\\
=&\int_{\Omega_{k_0}}B(\bu)\cdot(\zeta\bu)-
\int_{\Omega_{k_0}}\mathcal{P}\,\bu\cdot\nabla\zeta-\frac{1}{2}\int_{\Omega_{k_0}}
\bu\cdot \chi_{p^\prime}\cdot\nabla\zeta\, .
\end{array}
$$
Then, combining \eqref{chi zeta u} and \eqref{aukuk}, we have
$\lim_{k\to\infty}\langle A_{\zeta,k}\bu^k,\bu^k\rangle=\langle\chi_\zeta,\bu\rangle$, and thus
conclude the proof of Theorem \ref{main theorem}. 
\end{proof} 

\smallskip

Next, we make some remarks and prove two additional results, one
on the rate of dissipation of energy of the solution obtained for problem 
\eqref{lady-solo pb} and another on the uniqueness of
solution.   

\begin{remark} Dropping the convective term $\bv\cdot\nabla\bv$ in \eqref{lady-solo pb}$_1$, we obtain the {\em Ladyzhenskaya-Solonnikov problem for Stokes' system with a power law}. The solution of this problem can be obtained as in the proof of Theorem \ref{main theorem}, with obviously much less computations.
\end{remark}

The solution of problem \eqref{lady-solo pb} has energy dissipation
uniformly distributed along the outlets. More precisely, we have the
following result, which generalizes Theorem 3.2 in \cite{ls} for
power law shear thickening fluids.

\begin{prop}
\label{prop.2} 
Let $\bv$ be a solution of problem \eqref{lady-solo pb}, with $p\ge2$,
obtained by the proof of Theorem \ref{main theorem}. Then there exists a
constant $\kappa$ such that
\begin{equation}
\label{eq.82}
\int_{\Omega_i,t-1,t}|\nabla\bv|^p\leq
\kappa\,,\;\;\forall\,t \geq 1,
\end{equation}
where $i=1,2$.
\end{prop}
\begin{proof} Let $\bu=\bv-\ba$. By the proof of Theorem 
\ref{main theorem}, $\bu$ is the weak limit in 
$W_{\mbox{\tiny loc}}^{1,p}(\overline{\Omega})$ of a sequence
$\{\bu^k\}_{k=1}^\infty$, where $\bu^k$ is a solution of problem \eqref{truncated} with $T=k$. Now, for $\tau\ge\max\{t,2\}$ we define the
function
$$
z_\tau(\eta)=\int_{\eta-1}^\eta y_\tau(t)dt, \ \ \eta\ge1,
$$
where 
$$
y_\tau(t):=\frac{1}{k}\left|\bu^k\right|_{1,2,\Omega_{i,\tau-t,\tau+t}}^2+
\left|\bu^k\right|_{1,p,\Omega_{i,\tau-t,\tau+t}}^p
$$
(see Section \ref{section 2} for the definition of $\Omega_{i,\tau-t,\tau+t}$). Similarly to the proof of \eqref{estimate for y}, it is possible to show that
$$
z_\tau(\eta)\leq\varphi(\eta)\,,\;\;\forall\,\eta\in[1,\tau]\,,$$
where $\varphi(\eta)=
c_2\eta+c_3$, for some constants $c_2,c_3$. Since
$$y_\tau(1/2)=\int_{1/2}^{3/2}y_\tau(1/2)\,dt\leq
\int_{1/2}^{3/2}y_\tau(t)\,dt=z_\tau(3/2)\leq\varphi(3/2)\equiv c\,,
$$
we have
$$\dsp\int_{\tau-1/2}^{\tau+1/2}\int_{\Sigma_i}\left|\nabla\bu^k\right|^p\leq
y_\tau(1/2)\leq c
$$
and, consequently, by the weak convergence of $\bu^k$ to $\bu$, we also
have
$$\int_{\tau-1/2}^{\tau+1/2}\int_{\Sigma_i}\left|\nabla\bu\right|^p\leq
c\,,
$$
which is \eqref{eq.82} with $\bu$ in place of $\bv$. Since, by Lemma \ref{a}, the vector field $\ba$ also satisfies this property, this end
the proof of Proposition~\ref{prop.2}.
\end{proof}


\smallskip

In \cite[p.1437]{marusic}, Maru\v{s}i\'c-Paloka observes the difficult of
obtaining uniqueness results for Navier-Stokes system with a power law.
In particular, this is an open question even in bounded domains.
We can prove an uniqueness result for problem \eqref{lady-solo pb}
under some conditions, which we specify precisely in 
Theorem \ref{teo.2} below. One of these conditons is motivated by Proposition \ref{prop.2} and another, by
the following propostion, which was inspired by the solution of Leray
problem given by Maru\v{s}i\'c-Paloka; cf. \cite[Lemma 4.2/(4.24)]{marusic}. 

\begin{prop}
\label{lema.10} For $i$ either
equal to $1$ or $2$, let $\bv=(v_1,\cdots,v_n)$ be a divergence free vector
field in $W_{\mbox{\tiny loc}}^{1,p}(\overline{\Omega_i})$ vanishing
on $\partial\Omega_i$ and having property \eqref{eq.82}. 
If for some $j\in\{1,\ldots,n\}$ and some positive number~$c$,
\begin{equation}
\label{eq.84}
\left|\frac{\partial v_j}{\partial x_j}(x)\right|\geq
c|x^\prime|^{1/(p-1)}
\end{equation}
for all $x=(x_1,x')\in\Omega_i$, then there is a constant $C$ such that
\begin{equation}
\label{eq.85}
(\bw\cdot\nabla \bv,\bw)_{\Omega_{i,t}}
\leq C{\kappa} \,\|\, |D({\bv})|^{(p-2)/2}D(\bw)\|_{2,\Omega_{i,t}}^2,
\end{equation}
for all $\bw\in {\mathcal D}_{\mbox{\tiny loc}}^{1,p}(\Omega_i)$
and $t>0$.
\end{prop}
\begin{proof} Denote $\Omega=\Omega_i$. By H\"older inequality 
and \eqref{eq.82}, we obtain
\begin{equation}
|(\bw\cdot\nabla\bv,\bw)_{\Omega_{t-1,t}}|
\leq C{\kappa} \|\bw \|_{2p^\prime,\Omega_{t-1,t}}^2.
\label{eq.86}
\end{equation}
By Sobolev embedding and Poincar\'e inequality (Lemma \ref{inequalities}),
we have
\begin{equation}
\| \bw \|_{2p^\prime,\Omega_{t-1,t}}
\leq C |\bw|_{1,r,\Omega_{t-1,t}}\,,\label{eq.87}
\end{equation}
for any $r\in (1,2)$ such that $2p^\prime\leq \frac{rn}{n-r}$. Now, by
Korn inequality (Lemma \ref{inequalities}), H\"older inequality and \eqref{eq.84}, we obtain
$$
\begin{array}{rl}
&|\bw|_{1,r,\Omega_{t-1,t}}^r
=\int_{\Omega_{t-1,t}}|\nabla\bw|^r\\
\leq& C\int_{\Omega_{t-1,t}}|D(\bv)|^{(p-2)r/2}|D(\bw)|^r\frac{1}
{|D(\bv)|^{(p-2)r/2}}\\
\leq& C(\int_{\Omega_{t-1,t}}|D(\bv)|^{p-2}|D(\bw)|^2)^{r/2}(
\int_{\Omega_{t-1,t}}\frac{1}
{|D(\bv)|^{(p-2)r/(2-r)}})^{(2-r)/2}\\
\leq&C\||D(\bv)|^{(p-2)/2}D(\bw)\|_{2,\Omega_{t-1,t}}
^r(\int_0^{l_2}s^{n-2-\frac{(p-2)r}{(p-1)(2-r)}}ds)^{(2-r)/2}.
\end{array}
$$
Then, chosen $r\leq\frac{2(n-1)(p-1)}{np-(n+1)}$, it follows that
\begin{equation}
|\bw|_{1,r,\Omega_{t-1,t}}
\leq C\| |D(\bv)|^{(p-2)/2}D(\bw)\|_{2,\Omega_{t-1,t}}.
\label{eq.88}
\end{equation}
Thus, from $(\ref{eq.87})$ and $(\ref{eq.88})$, we have
$$
\|\bw\|_{2p^\prime,\Omega_{t-1,t}}
\leq C\| |D(\bv)|^{(p-2)/2}D(\bw)\|_{2,\Omega_{t-1,t}}
$$
and from $(\ref{eq.86})$, we get
\begin{equation}
|(\bw\cdot\nabla\bv,\bw)_{\Omega_{t-1,t}}|
\leq C{\kappa}\, \|\, |D(\bv)|^{(p-2)/2}D(\bw)\|_{2,\Omega_{t-1,t}}^2.
\label{eq.89}
\end{equation}
Finally, writing $\Omega_t$ as a finite union of domains
$\Omega_{t-j-1,t-j}$,\break\hfill $j=0,\cdots,m<~\infty$, and adding inequality \eqref{eq.89} with $\Omega_{t-j-1,t-j}$ in place of 
$\Omega_{t-1,t}$ with respect to $j$, we obtain \eqref{eq.85}.
\end{proof}

\begin{remark}
\label{obs.7} An example of a solution satisfying property
\eqref{eq.84} when $\Omega_i$ is a straight outlet (i.e. the cross
sections $\Sigma(x_1)$ are constant, with respect to $x_1$) is the Poiseuille flow in $\Omega_i$. See \cite[\S 3]{marusic}.
\end{remark}

We now state and prove our uniqueness result.

\begin{teo}
\label{teo.2} Let $\kappa>0$ be sufficiently small and $l$ be some
positive number. Then there is no more than one weak solution of problem
\eqref{lady-solo pb} in $W^{2,l}_{\mbox{\tiny loc}}(\overline{\Omega})$
and satisfying \eqref{eq.82} and property \eqref{eq.85} in $\Omega$,
i.e. for some constant~$C$, 
\begin{equation}
\label{eq.85 in Omega}
(\bw\cdot\nabla \bv,\bw)_{\Omega,t}
\leq C{\kappa} \,\|\, |D({\bv})|^{(p-2)/2}D(\bw)\|_{2,\Omega_{t}}^2,
\end{equation}
for all $\bw\in {\mathcal D}_{\mbox{\tiny loc}}^{1,p}(\Omega)$ and $t>0$.
\end{teo}

\begin{proof} Let ${\bv}_1$ and ${\bv}_2$ be solutions of 
\eqref{lady-solo pb} satisfying the assumptions in Theorem \ref{teo.2}.
Denote ${\bw}={\bv}_1-{\bv}$. Then
$$\begin{array}{l}
-\mbox{div}\{D({\bw})+|D({\bw})+D({\bv}_2)|^{p-2}[D({\bw})+
D({\bv}_2)]-|D({\bv}_2^k)|^{p-2}D(\bv_2)\}\\+
{\bw}\cdot\nabla{\bw}+{\bw}\cdot\nabla\bv_2+{\bv}_2\cdot\nabla{\bw}
+\nabla\left(\mathcal{P}_1-\mathcal{P}_2\right)=0,
\end{array}
$$
where $\mathcal{P}_1,\mathcal{P}_2\in W^{1,l}_{\mbox{\tiny loc}}(\overline{\Omega})$. Multiplying this equation by $\bw$ and integrating by parts
over $\Omega_t$, similarly to derivation of \eqref{integrating by parts}, we obtain   
$$
\begin{array}{c}
\int_{\Omega_t}\left\{|D({\bw})+D(\bv_2)|^{p-2}[D({\bw})+D(\bv_2)]-|D({\bv}_2)|^{p-2} D(\bv_2)\right\}:D({\bw})\\
=-({\bw}\cdot\nabla {\bv}_2,{\bw})_{\Omega_t}-I\,,
\end{array}
$$
where
$$\begin{array}{c}
I=-\int_{\partial\Omega_t}\frac{\left|{\bw}\right|^2}{2}({\bw}\cdot{\bn} +\bv_2\cdot{\bn} -\left(\mathcal{P}_1-\mathcal{P}_2\right)
\left({\bw}\cdot{\bn}\right)\\
+\int_{\partial\Omega_t}{\bw}\cdot\{|D({\bw})+D(\bv_2)|^{p-2} [D({\bw})+D(\bv_2)]-|D({\bv}_2)|^{p-2}D({\bv}_2)\}\bn\,. 
\end{array}
$$
Estimating some terms in the above equation by using Lema \ref{inequalities} and assumption \eqref{eq.85 in Omega}, it follows that 
$$
c_1\int_{\Omega_t}|\nabla{\bw}|^p+c_2\int_{\Omega_t}|D(\bv_2)|^{p-2}|D({\bw})^2
\leq c_3|\kappa|\int_{\Omega_t}|D(\bv_2)|^{p-2} \left|D({\bw})\right|^2+I\,,
$$
for some positive constants $c_1,c_2,c_3$. 
Thus, if $|\kappa|<c_1/c_2$, we have
$$
y(t):=|{\bw}|_{1,p,\Omega_t}^p y(t)\leq cI\,,
$$
(for some constant $c$). Next, integrating $y(t)$ from $\eta-1$ to $\eta$,
$\eta\ge1$, and proceeding similarly to the proof of \eqref{second estimate for z}, but using \eqref{eq.82} instead of Lemma \ref{a}\,ii), 
we obtain
$$
z(\eta)\leq c\Psi\left(z^\prime(\eta)\right),$$
with $\Psi(\tau)=\tau+\tau^{1/p}+\tau^{2/p}+\tau^{3/p}$. 
Now suppose $z$ is not identically zero. Then, by Lema \ref{ls lemma},
we have
$$\lim_{t\rightarrow\infty}z(t)=\infty.$$
Besides, since for $\tau\geq\tau_1$ (for some $\tau_1>0$)
$$\Psi(\tau)\leq
\left\{ \begin{array}{rll} \tau,&\mbox{if}&p\geq3\\
\frac{3}{p},&\mbox{if}&p<3\end{array}\right. ,$$  
by Lemma \ref{ls lemma} again, we also have
$$\begin{array}{lll}
\liminf_{t\rightarrow\infty}\mbox{e}^{-t}z(t)>0,& \mbox{if} &p\geq3\\
\liminf_{t\rightarrow\infty}t^{-3/(3-p)}z(t)>0,& \mbox{if} &p<3\,.
\end{array}$$
This contradicts \eqref{lady-solo pb}$_5$.
Therefore, $z\equiv0$ and so, $\bv_1=\bv_2$.
\end{proof}

\begin{remark}
\label{kappa alpha} By tracking all the estimates we did to obtain \eqref{estimate for y}, similarly, to obtain \eqref{eq.82}, we can
see that the constant $\kappa$ in \eqref{eq.82} depends on the
flux $\alpha$ so that $\kappa={\mathcal O}(|\alpha|^\gamma)$, for
some positive number $\gamma$. In particular, $\kappa$ tends to zero
when $\alpha$ tends to zero. 
\end{remark}

\begin{remark}
For an example where condition \eqref{eq.85 in Omega} is accomplished, see
\cite[p.1437/\S 4.2]{marusic}.
\end{remark}

\begin{remark} Regarding the Stokes system with a power law, i.e.
system \eqref{nspl} discarding $(\bv\cdot\nabla)\bv$, we have uniqueness
of solution for the corresponding {Ladyzhenskaya-Solonnikov problem}
for any flux $\alpha$, as occurs in the case $p=2$ \cite[Corollary 2.1, p. 739]{ls}. 
\end{remark}

\end{document}